\documentclass[a4paper]{amsart}

\usepackage{amsmath,amssymb}
\usepackage[dvipsnames]{xcolor}
\usepackage{pictexwd}
\usepackage{mathtools}
\usepackage{hyperref}

\newtheorem{theorem}{Theorem}[section]
\newtheorem{proposition}[theorem]{Proposition}
\newtheorem{corollary}[theorem]{Corollary}

\newtheorem{preremark}[theorem]{Remark}
\newtheorem{predefinition}[theorem]{Definition}
\newtheorem{preexample}[theorem]{Example}
\newtheorem{prenotation}[theorem]{Notation}
\newtheorem{preconjecture}[theorem]{Conjecture}

\newenvironment{remark}{\begin{preremark}\rm}{\end{preremark}}
\newenvironment{definition}{\begin{predefinition}\rm}
{\end{predefinition}}
\newenvironment{example}{\begin{preexample}\rm}{\end{preexample}}

\allowdisplaybreaks
\def\cmat{{\mathop{\rm cmat}\nolimits}}

\def\OO{{\mathcal{O}}}

\def\AA{\mathbb{A}}

\def\ZZ{\mathbb{Z}}
\def\QQ{\mathbb{Q}}

\def\PP{\mathbb{P}}
\def\TT{\mathbb{T}}

\newcommand{\M}{{\mathfrak{M}}}

\newcommand{\m}{{\mathfrak{m}}}

\let\epsilon=\varepsilon

\def\phi{{\varphi}}
\let\Psi=\varPsi
\let\Phi=\varPhi
\let\theta=\vartheta
\let\rho=\varrho

\def\Mat{\mathop{\rm Mat}\nolimits}

\def\Rad{\mathop{\rm Rad}\nolimits}
\def\Ann{\mathop{\rm Ann}\nolimits}

\def\Spec{\mathop{\rm Spec}\nolimits}
\def\Proj{\mathop{\rm Proj}\nolimits}
\def\Cot{\mathop{\rm Cot}\nolimits}

\def\rk{\mathop{\rm rk}\nolimits}

\def\Sing{\mathop{\rm Sing}\nolimits}
\def\Jac{\mathop{\rm Jac}\nolimits}

\def\ecod{\mathop{\rm ecod}\nolimits}

\newcommand{\Lin}{\mathop{\rm Lin}\nolimits}
\newcommand{\edim}{\mathop{\rm edim}\nolimits}

\newcommand{\BO}{\mathbb{B}_{\mathcal{O}}}

\newcommand{\id}{\mathop{\rm id}\nolimits}
\newcommand{\Hilb}{\mathop{\rm Hilb}\nolimits}

\let\To=\longrightarrow

\def\TTTo#1\mathop{\xrightarrow{\hspace*{1cm}}^{^{\mkern-70mu (#1}}}

\def\tr{^{\,\rm tr}}
\def\tfrac #1#2{{\textstyle\frac{#1}{#2}}}
\def\tsum_#1^#2{{\textstyle\sum\limits_{#1}^{#2}}}

\definecolor{darkred}{cmyk}{0.0, 0.7, 1.0, 0.0}

\def\cocoa{\mbox{\rm
  C\kern-.13em o\kern-.07 em C\kern-.13em o\kern-.15em A}}
\def\apcocoa{\mbox{\rm
A\kern-0.13em p\kern -0.07em C\kern-.13em o\kern-.07 em C\kern-.13em
o\kern-.15em A}}

\begin{document}

\title{Graded Algebras over Polynomial Rings}

%    Information for the first author
\author{Martin Kreuzer}
\address{Fakult\"at f\"ur Informatik und Mathematik, Universit\"at
Passau, D-94030 Passau, Germany}
\email{martin.kreuzer@uni-passau.de}

%    Information for the second author
\author{Lorenzo Robbiano}
\address{Dipartimento di Matematica, Universit\`a di Genova,
Via Dodecaneso 35,
I-16146 Genova, Italy}
\email{lorobbiano@gmail.com}

\date{\today}

\begin{abstract}
Given a trivially graded polynomial ring $P_0=K[a_1,\dots,a_m]$ over a field~$K$
and a positively graded polynomial ring $P=P_0[x_1,\dots,x_s]$, we study graded rings
$R=P/I$, where~$I$ is a homogeneous ideal of~$P$ satisfying $I\cap P_0=\{0\}$.
The induced morphism
$\Theta\colon\Spec(R)\To \Spec(P_0)=\AA^m_K$
is used to prove that $\Spec(R)$ is connected.
Next, we characterize and compute the following loci in~$\AA^m_K \colon$ the set
$\Sing_0(\Theta)$ of all points whose corresponding point in the zero section of~$\Theta$
is singular in~$\Spec(R)$, the set
$\Sing_v(\Theta)$ of all points~$\Gamma$ such that the origin of the fiber
$F_\Gamma$ of~$\Theta$ is singular, and the set
$\Sing_s(\Theta)$ of all points~$\Gamma$ such that
$\dim(\Sing(F_\Gamma))\ge1$.
Finally, we apply these results to MaxDeg border basis schemes.
Their coordinate rings are non-negatively graded by the total arrow degree and
satisfy the above hypotheses.
In particular, we explicitly determine the singular loci of the
$\OO$-border basis schemes with $\OO=\{1,x,y,z,z^2\}$ and $\OO=\{1,x,y,z,yz\}$.
\end{abstract}



% 13A02 Graded Rings
% 13E10 finite dimensional algebras
% 13P10 Gröbner bases; other bases for ideals and modules

% 14D20 Algebraic moduli problems
% 14E25 Embeddings in algebraic geometry
% 14Q20 Computational Aspects of Alg Geom
% 14R10 Affine Spaces

\maketitle

%\tableofcontents

\bigbreak
%%%%%%%%%%%%%%%%%%%%%%%%%%%%%%%%%%%%%%
%
% Section 1: Introduction
%
%%%%%%%%%%%%%%%%%%%%%%%%%%%%%%%%%%%%%%

\section{Introduction}
\label{sec-Intro}

In a series of papers \cite{KLR0,KLR1,KLR2,KLR3,KR5,KR6,Rob1,Rob2}, 
the present authors and L.N.\ Long developed computational techniques for analyzing border basis schemes. 
A particularly suitable class for this endeavour are MaxDeg border basis schemes, since their coordinate 
rings are non-negatively graded with respect to the total arrow degree. For instance, in~\cite{KR6}, we succeeded in
constructing optimal reembeddings of such border basis schemes into lower dimensional affine spaces
and in determining cases when these schemes are isomorphic to an affine space.

Here we continue this effort, but from a more general point of view. More precisely, we study
non-negatively graded algebras of the following form. Let $K$ be a field, let $P_0=K[a_1,\dots, a_m]$ be a trivially graded
polynomial ring over~$K$, let $P=P_0[x_1,\dots,x_s]$ be a positively graded polynomial ring over~$P_0$, and let
$I\subseteq P$ be a homogeneous ideal with $I \cap P_0 = \{0\}$. In this case, we say that $P/I$
is a {\it positive $P_0$-algebra}. The coordinate ring of a MaxDeg border basis scheme is an example of such an algebra.

Geometrically, a positive $P_0$-algebra $R$ gives rise  to a morphism of schemes 
$\Theta\colon\Spec(R) \To \AA^m_K$ which has a zero section and whose fibers are weighted projective schemes.
In Section~\ref{sec-Connectedness} we prove that the affine scheme $\Spec(R)$
associated to a positive $P_0$-algebra is connected (see Proposition~\ref{prop-Rconnected}).
In particular, this implies that MaxDeg border basis schemes are connected.
It is a classical result of R.\ Hartshorne that the Hilbert schemes $\Hilb^\mu(\AA^n_K)$ are connected (see~\cite{Har,PS}).
Although border basis schemes can be embedded as open subschemes of such Hilbert schemes, it is not known
in general whether all border basis schemes are connected (see~\cite{Rob1,Rob2}).

In~\cite[Theorem~5.9]{KR5} we showed that, if the point $\Gamma_0$ in the zero section
of~$\Theta$ lying above $\Gamma\in K^m$ is regular, then  the vertex of the 
fiber~$F_\Gamma$ of~$\Theta$ is also a regular point of the fiber, and this implies that the
fiber is isomorphic to an affine space.

Thus we are led to introduce the following
three sets of points in Section~\ref{sec-Singularities}:
\begin{align*}
&\Sing_0(\Theta) = \{ \Gamma\in K^m \mid \Gamma_0 \in \Sing(R) \}  
&& \hbox{\rm (set of zero section singularities)} \\
&\Sing_v(\Theta) = \{ \Gamma \in K^m \mid (0,\dots,0) \in \Sing(F_\Gamma) \} 
&& \hbox{\rm (set of fiber vertex singularities)} \\
&\Sing_s(\Theta) = \{ \Gamma \in K^m \mid \{ (0,\dots,0) \} \subsetneq \Sing(F_\Gamma) \}
&& \hbox{\rm (set of singular fibers)}
\end{align*}
The fundamental inequalities established in Proposition~\ref{prop-fundamental} allow
us to connect the embedding codimensions of the local rings at the corresponding points
of the zero section in~$\Spec(R)$ and of the fiber~$F_\Gamma$. 
They imply the inclusion $\Sing_v(\Theta) \subseteq
\Sing_0(\Theta)$ (see Corollary~\ref{cor-singsing}) and yield characterizations of the regularity 
of~$R_{\Gamma_0}$ and $R_{F_\Gamma}$ (see Corollaries~\ref{cor-regR} and~\ref{cor-regF}).

In order to compute the above loci, we introduce additional tools
in Section~\ref{sec-AlinCoeffMat}. The $P_0$-linear part $\Lin_{P_0}(f)$ of a polynomial
$f\in P=P_0[x_1,\dots,x_s]$ is used to define the $P_0$-linear part of~$I$
via $\Lin_{P_0}(I)=\langle \Lin_{P_0}(f)\mid f\in I\rangle_{P_0}$.
We show that one can compute this $P_0$-module from any set of generators of~$I$
(see Proposition~\ref{prop-indepOfGen}).
Next, we encode $\Lin_{P_0}(I)$ using the $P_0$-linear coefficient matrix
$\cmat_{P_0}(G)$, whose entries are the coefficients in~$P_0$ of the indeterminates
$x_1,\dots,x_s$ in a system of generators~$G$ of~$I$.
Substituting a point $\Gamma\in K^m$ into~$\cmat_{P_0}(G)$ yields a relation matrix
$\cmat_K(G_\Gamma)$ of the conormal module of the fiber~$F_\Gamma$ at its origin.

The $P_0$-linear coefficient matrix is then used in
Section~\ref{sec-CompSing} to characterize the set of zero section
singularities $\Sing_0(\Theta)$ in terms of the ranks of the matrices
$\cmat_K(G_\Gamma)$ (see Theorem~\ref{thm-CharSing0}).
As a consequence, we obtain an algorithm for computing
$\Sing_0(\Theta)$ (see Corollary~\ref{cor-CompSing0}).
Moreover, we characterize the set of vertex singularities
$\Sing_v(\Theta)$ in terms of the ranks of the matrices
$\cmat_K(G_\Gamma)$.
Since this characterization involves the dimensions of the fibers, the
sets $\Sing_v(\Theta)$ are constructible subsets of~$\AA^m_K$, but are
not closed in general.
Nevertheless, using comprehensive Gr\"obner bases, we construct
algorithms for computing $\Sing_v(\Theta)$ and $\Sing_s(\Theta)$
(see Corollaries~\ref{cor-CompSingV} and~\ref{cor-CompSingS}).
Except in very simple cases, these computations appear to be beyond the
capabilities of current computer algebra systems.

In the final section we return to the setting that originally sparked our interest
in positive $P_0$-algebras, namely the study of MaxDeg border basis schemes. For planar
MaxDeg border basis schemes over a perfect field~$K$, i.e., for border basis schemes
$\BO$ associated with an order ideal~$\OO$ in two indeterminates, we showed
in~\cite[Corollary~3.6]{KR6} that~$\BO$ is an affine cell and hence non-singular.
The next case is the border basis scheme corresponding to
$\OO = \{1,x,y,z\}$, whose coordinate ring is positively graded and for which
$\Sing_0(\Theta)=\Sing_v(\Theta)=\Sing_s(\Theta)=\{0\}$.

Significantly more interesting are the two types of spatial order ideals of length~5.
For $\OO=\{1,x,y,z,z^2\}$, the morphism
$\Theta\colon\BO\longrightarrow\AA^5_K$
can be simplified via a suitable $Z$-separating reembedding such that the image of~$\BO$
is contained in~$\AA^{21}_K$.
In Example~\ref{ex-13z2}, we show that the corresponding
$P_0$-linear coefficient matrix is block diagonal, with blocks of sizes $3\times13$ and
$9\times3$.
The ideal of minors defining $\Sing_0(\Theta)$
turns out to have the simple radical
$\mathfrak{p}=\langle c_{53}-c_{51}^2,\,
c_{54}-c_{51}c_{52},\, c_{55}-c_{52}^2\rangle$,
which shows that $\Sing_0(\Theta)$ is isomorphic to an affine plane.
Additionally, we verify that $\Sing_v(\Theta)$ and $\Sing_s(\Theta)$ coincide.
Notice that the monomial point of this border basis scheme is singular.

The second type of order ideal is represented by
$\OO=\{1,x,y,z,yz\}$ and examined in Example~\ref{ex-13yz}.
Here the monomial point of~$\BO$ is regular, but otherwise the computation proceeds
in a very similar fashion to the preceding example.
All three singular loci turn out to be defined by the same prime ideal
$\mathfrak{p}=\langle c_{52}-c_{51}c_{54},\, c_{55}-c_{51}c_{54}^2,\,
c_{51}c_{53}-1\rangle$.
A final observation is that, in both cases, the generic fiber of~$\Theta$ is
10-dimensional, whereas the singular fibers are 11-dimensional.

\medskip
Throughout this paper we follow the notation and terminology introduced
in~\cite{KR1,KR2}. As far as the specific notation for border bases and border basis schemes is
concerned, we remain consistent with our previous works
\cite{KLR0,KLR1,KLR2,KLR3,KR3,KR4,KR5,KR6}.
Regarding algebraic geometry, we allow ourselves the following conventions.
When we write $\Spec(R)$, we actually mean the scheme
$(\Spec(R),\mathcal{O}_{\Spec(R)})$.
When we refer to closed points of a scheme defined over a field~$K$, we assume that they
are $K$-rational. In principle, one should consider points with coordinates in the algebraic
closure of~$K$, but in the cases studied here, we believe that this slight abuse of terminology
causes no problems.

All results are based on extensive computations of explicit examples using the
computer algebra system \cocoa\ (see~\cite{ABR}). The source files for the examples are available
from the authors upon request.

\bigskip\bigbreak
%%%%%%%%%%%%%%%%%%%%%%%%%%%%%%%%%%%%%%%%%%%%%%%%
%
% Section 2: Definitions and Basic Properties 
%
%%%%%%%%%%%%%%%%%%%%%%%%%%%%%%%%%%%%%%%%%%%%%%%%

\section{Positive  \texorpdfstring{$P_0$}{P0}-Algebras}

\label{sec-PosAlg}

As explained in the introduction, let $K$ be a field, and let
$P=P_0[x_1,\dots,x_s]$ be a polynomial ring over~$P_0$, where
$P_0=K[a_1,\dots,a_m]$ is a polynomial ring over~$K$ which we regard as
trivially graded, i.e., $\deg(a_i)=0$ for $i=1,\dots,m$.
The ring~$P$ is equipped with an additional grading defined by a weight vector
$W=(w_1,\dots,w_s)\in\ZZ^s$.
We denote the tuple $(x_1,\dots,x_s)$ by~$X$.
Finally, we set $n=m+s$, so that
$P=P_0[X]$ is a polynomial ring in~$n$ indeterminates.

Our main objective is to study the following class of algebras.

\begin{definition}\label{del-posalg}
Let $P_0$ and $P$ be as above.
A graded $P_0$-algebra is a ring of the form $P/I$, where $I$ is a $W$-homogeneous
ideal of~$P$.
If $w_j>0$ for $j=1,\dots,s$ and $I\cap P_0=\langle0\rangle$, we say that $P/I$
is a \textbf{positive $P_0$-algebra}.
Equivalently, $P/I$ is a positively graded $P_0$-algebra
with $(P/I)_0=P_0$, so that the canonical $K$-algebra homomorphism
$\phi\colon P_0\to P/I$ is injective.

In the special case $P_0=K$, a positive $K$-algebra $P/I$ satisfies
$I\cap K=\langle0\rangle$ if and only if the  ideal $I$ is proper.
\end{definition}

The following remark collects some basic properties of positive $P_0$-algebras.

\begin{remark}\label{rem-special}
Assume that $R=P/I$ is a positive $P_0$-algebra.
\begin{enumerate}
\item[(a)] Geometrically,  we get a closed subscheme
$\Spec(R)\subseteq\AA^n_K =\Spec(P)$.

\item[(b)] The ideal~$I$ is contained in
$\langle X\rangle=\langle x_1,\dots,x_s\rangle$.
Consequently,  the canonical $K$-algebra homomorphism
$\theta\colon P_0\To R$ is injective.

\item[(c)] Letting $\AA^m_K=\Spec(P_0)$, the homomorphism~$\theta$
corresponds to a surjective morphism
$\Theta\colon\Spec(R)\longrightarrow\AA^m_K$
(see~\cite[Rem.~5.2]{KR5}) which, at the level of $K$-rational points,
is given by the projection onto the first~$m$ coordinates.
The morphism~$\Theta$ is called the \textbf{canonical morphism},
and $\AA^m_K$ is called its \textbf{base space}.
\end{enumerate}
\end{remark}

One of our main goals is to study the fibers of the canonical morphism~$\Theta$.
To this end, we extend and slightly modify the notation introduced in~\cite[Section~5]{KR5}.

\begin{definition}\label{def-fiber}
Let $R=P/I$ be a positive $P_0$-algebra, let
$\Theta\colon\Spec(R)\longrightarrow\AA^m_K$
be the canonical morphism, and let
$\Gamma=(\gamma_1,\dots,\gamma_m)\in K^m$.
\begin{enumerate}
\item[(a)] The point
$\Gamma_0=(\gamma_1,\dots,\gamma_m,0,\dots,0)\in K^n$
is called the \textbf{zero point} over the base point~$\Gamma$.

\item[(b)] The maximal ideal of~$P$ corresponding to the zero point over~$\Gamma$
is denoted by~$\M_{\Gamma_0}$.

\item[(c)] The local ring $R_{\M_{\Gamma_0}/I} = (P/I)_{\M_{\Gamma_0}/I}$ of the zero point over~$\Gamma$ is
denoted by~$R_{\Gamma_0}$. Its maximal ideal $(\M_{\Gamma_0}/I)\cdot R_{\M_{\Gamma_0}/I}$ 
is denoted by~$\m_{\Gamma_0}$ and its dimension by~$d_{\Gamma_0}$.

\item[(d)] The closed subscheme
$F_\Gamma=\Theta^{-1}(\Gamma)$ of~$\Spec(R)$
is called the \textbf{fiber} of~$\Theta$ over~$\Gamma$.

\item[(e)] The ideal
$I_\Gamma=
I\cdot(P/\langle a_1-\gamma_1,\dots,a_m-\gamma_m\rangle)
\cong
I|_{a_1\mapsto\gamma_1,\dots,a_m\mapsto\gamma_m}\cdot K[X]$
is called the \textbf{fiber ideal} of~$\Theta$ over~$\Gamma$.
The ring $K[X]/I_\Gamma$ is the coordinate ring of~$F_\Gamma$.
Its maximal ideal at the origin is
$\langle X\rangle/I_\Gamma$.
Via this identification, we may regard $F_\Gamma$ as a closed subscheme of~$\AA^s_K$.

\item[(f)] The local ring
$(K[X]/I_\Gamma)_{\langle X\rangle/I_\Gamma}$
of~$F_\Gamma$ at the origin is denoted by~$R_{F_\Gamma}$.
Its maximal ideal is denoted by~$\m_{F_\Gamma}$,
and its dimension by~$d_{F_\Gamma}$.
\end{enumerate}
\end{definition}

Intuitively, the morphism~$\Theta$ can be viewed as a \textit{family of fibers} over $\AA^m_K$.
The zero points of the fibers fit together to form the following section.

\begin{definition}\label{def-ZeroSection}
The morphism
$\zeta\colon\AA^m_K\longrightarrow\Spec(R)$ which is
induced by the canonical $K$-algebra epimorphism
$R=P/I\longrightarrow P/\langle x_1,\dots,x_s\rangle\cong P_0$
is called the \textbf{zero section morphism} of~$\Theta$.
Its image is called the \textbf{zero section} of~$\Theta$ and is denoted by~$Z_R$.
\end{definition}

\begin{remark}\label{rem-Section}
The zero section morphism~$\zeta$ is indeed a section of~$\Theta$, since we have 
$\Theta\circ\zeta=\id_{\AA^m_K}$.
Consequently, $\zeta$ maps $\AA^m_K$ isomorphically onto the zero section~$Z_R$.
Moreover, for every base point~$\Gamma\in K^m$, we have
$\zeta(\Gamma)=\Gamma_0$.
\end{remark}

A natural question is whether the zero section morphism~$\zeta$ is the only section of~$\Theta$.
If there exists another section, then on the level of coordinate rings it is given by
a $P_0$-algebra homomorphism
$R\longrightarrow P/\langle x_1-\alpha_1,\dots,x_s-\alpha_s\rangle\cong P_0$
for some $(\alpha_1,\dots,\alpha_s)\in K^s\setminus\{0\}$.
Necessarily, we have
$I\subseteq\langle x_1-\alpha_1,\dots,x_s-\alpha_s\rangle$.
At first glance, this may seem unlikely for a $W$-homogeneous ideal~$I$.
Nevertheless, the following example shows that non-trivial sections can indeed occur.

\begin{example}\label{ex-moresections}
Let $P_0 = \QQ[a,b]$, and let $P = P_0[x_1,x_2,x_3]$ be graded by $W=(1,1,1)$.
Then the ideal $I = \langle ax_1 -bx_2 +(b-a) x_3 \rangle$ is $W$-homogeneous, and
$R=P/I$ is a positive $P_0$-algebra. 

Now we consider the ideal $J = \langle x_1-1,\ x_2-1, x_3-1\rangle$ and note that $I \subseteq J$.
Hence we obtain an epimorphism  $\sigma \colon R = P/I \longrightarrow P/ J \cong P_0$, 
and this yields a section $\Spec(\sigma)\colon\AA^2_\QQ \longrightarrow \Spec(R)$ of~$\Theta$.
\end{example}

\bigbreak
%%%%%%%%%%%%%%%%%%%%%%%%%%%%%%%%%%%%%%%%%%%%%%%%
%
% Section 3: Connectedness
%
%%%%%%%%%%%%%%%%%%%%%%%%%%%%%%%%%%%%%%%%%%%%%%%%

\section{Connectedness}
\label{sec-Connectedness}

As before, let $P_0=K[a_1,\dots,a_m]$ be trivially graded, let
$P=P_0[x_1,\dots,x_s]$ be positively graded by~$W$, and let
$R=P/I$ be a positive $P_0$-algebra.
For every point $\Gamma=(\gamma_1,\dots,\gamma_m)\in K^m$,
the coordinate ring $K[X]/I_\Gamma$ of the fiber~$F_\Gamma$
is positively graded by~$W$.
Hence the fiber~$F_\Gamma$ can be viewed as a closed subscheme of a weighted projective space.

Let us recall a few basic facts about weighted projective spaces.
(For more information, see for instance~\cite{BR}.)
Given a tuple of positive integers $W=(w_1,\dots,w_s)$
which defines a grading on $K[x_1,\dots,x_s]$ via
$\deg_W(x_i)=w_i$, the scheme
$\PP_K(W)=\Proj(K[x_1,\dots,x_s])$
is called the \textbf{$W$-weighted projective space}.

In the special case $W=(1,\dots,1)$, the space $\PP_K(W)$ is the ordinary
projective space $\PP_K^{s-1}$.
Recall that $\PP_K^{s-1}$ can be obtained as the quotient of the punctured
affine space $\AA_K^s\setminus\{0\}$ by the lines through the origin with
the origin removed, i.e., the lines $L_\pi$ for
$\pi=(\pi_1,\dots,\pi_s)\in K^s\setminus\{0\}$ defined parametrically by
$x_1=\pi_1t,\dots,x_s=\pi_st$ with $t\in K\setminus\{0\}$.

Similarly, the weighted projective space $\PP_K(W)$ can be obtained as the
quotient of $\AA_K^s\setminus\{0\}$ by the punctured rational curves $C_\pi$
for $\pi=(\pi_1,\dots,\pi_s)\in K^s\setminus\{0\}$ defined parametrically by
$x_1=\pi_1t^{w_1},\dots,x_s=\pi_st^{w_s}$ with
$t\in K\setminus\{0\}$.

\smallskip
Using this description, we now show that $\Spec(R)$ is connected.
This fact was mentioned, but not proved, in~\cite{KR5}.

\begin{proposition}{\bf (Connectedness of
\texorpdfstring{$\Spec(R)$}{Spec(R)})}\label{prop-Rconnected}$\mathstrut$\\
Let $P_0=K[a_1,\dots,a_m]$ be trivially graded,
let $P=P_0[x_1,\dots,x_s]$ be positively graded by
$W=(w_1,\dots,w_s)$, and let $R=P/I$ be a positive $P_0$-algebra.
\begin{enumerate}
\item[(a)] Let $\Gamma=(\gamma_1,\dots,\gamma_m)\in K^m$ be a base point
and $\Gamma_0=(\gamma_1,\dots,\gamma_m,0,\dots,0)$
the zero point over~$\Gamma$. Moreover, assume that
$\Pi=(\gamma_1,\dots,\gamma_m,\pi_1,\dots,\pi_s)$
is a $K$-ra\-tio\-nal point on the fiber~$F_\Gamma$.
Let $t$ be a new indeterminate, and let
$\psi\colon K[X]\longrightarrow K[t]$
be the $K$-algebra homomorphism defined by
$\psi(x_i)=\pi_i\,t^{w_i}$ for $i=1,\dots,s$.
Then the kernel of~$\psi$ defines an irreducible curve
$C_{\Pi,\Gamma_0}\subseteq\AA^s_K$
which is contained in~$F_\Gamma$ and contains both points~$\Pi$ and~$\Gamma_0$.

\item[(b)] The fiber $F_\Gamma$ of~$\Theta$ is connected.

\item[(c)] The scheme $\Spec(R)$ is connected.
\end{enumerate}
\end{proposition}

\begin{proof}
First we prove~(a). The $K$-algebra homomorphism
$\psi\colon K[X]\longrightarrow K[t]$ is homogeneous with respect to the
grading given by~$W$ on~$K[X]$ and the standard grading on~$K[t]$. 
Since its image is an integral domain, its kernel is a prime ideal. 
Hence the scheme $C_{\Pi,\Gamma_0}$ is an irreducible
curve in~$\AA^s_K$. Moreover, the preimages of the ideals
$\langle t-1\rangle$ and $\langle t\rangle$ under~$\psi$ show that
$\Pi$ and~$\Gamma_0$ belong to~$C_{\Pi,\Gamma_0}$.

It remains to prove that $C_{\Pi,\Gamma_0}\subseteq F_\Gamma$.
Let $f\in I_\Gamma$ be a $W$-homogeneous polynomial of $W$-degree~$d$.
Dividing $f$ by the ideal
$\langle x_1-\pi_1t^{w_1},\dots,x_s-\pi_st^{w_s}\rangle$
yields the remainder
$f(\pi_1t^{w_1},\dots,\pi_st^{w_s})$.
Since $f$ is homogeneous of degree~$d$, we have
\[
f(\pi_1t^{w_1},\dots,\pi_st^{w_s})
= t^d f(\pi_1,\dots,\pi_s)=0,
\]
because $\Pi\in F_\Gamma$.
Hence every polynomial in~$I_\Gamma$ vanishes on~$C_{\Pi,\Gamma_0}$, and therefore
$C_{\Pi,\Gamma_0}\subseteq F_\Gamma$.

To prove~(b), we may extend the base field, since connectedness is preserved
under field extensions. Thus we may assume that~$K$ is algebraically closed.
Let $\Pi,\Pi'$ be two closed points of~$F_\Gamma$.
The irreducible curve connecting~$\Pi$ and~$\Gamma_0$
is contained in an irreducible component of~$F_\Gamma$.
Likewise, the corresponding curve connecting~$\Pi'$ and~$\Gamma_0$
is contained in an irreducible component of~$F_\Gamma$.
Hence every irreducible component of~$F_\Gamma$ contains~$\Gamma_0$,
and therefore $F_\Gamma$ is connected.

Finally, we prove~(c).
Again, we may assume that~$K$ is algebraically closed.
Let $\Pi_1,\Pi_2$ be two closed points of~$\Spec(R)$.
By~(a), the point~$\Pi_1$ lies in the same connected component as the zero point
$\Gamma_0^{(1)}$ of its fiber, and similarly $\Pi_2$ lies in the same connected
component as the zero point $\Gamma_0^{(2)}$ of its fiber.

The morphism~$\Theta$ maps these zero points to the corresponding base points
$\Gamma^{(1)},\Gamma^{(2)}\in\AA^m_K$.
In the affine space there is a line joining
$\Gamma^{(1)}$ and~$\Gamma^{(2)}$.
Applying the zero section morphism~$\zeta$ to this line yields a connected subset
of~$\Spec(R)$ containing both $\Gamma_0^{(1)}$ and~$\Gamma_0^{(2)}$.
Consequently, $\Pi_1$ and~$\Pi_2$ belong to the same connected component of~$\Spec(R)$,
and therefore $\Spec(R)$ is connected.
\end{proof}

The following example illustrates the proposition.

\begin{example}\label{ex-ConnecingFibers}
Let $P_0=\QQ[a_1,a_2]$, and let $P = P_0[x_1,\dots, x_{12}]$ 
be graded by $W = (2,2,3,3, 1,1,2,2, 1,1,2,2)$ as a $P_0$-algebra. Then the ideal
$I=\langle f_1, \dots, f_6 \rangle$ generated by the polynomials 
$$
\begin{array}{lll}
f_1 = x_3  -a_1x_4  -x_2x_9 +x_1x_5, & f_2 =  x_7 -a_1x_8  -x_6x_{10}, \cr
f_3 = x_{12}  -a_2x_{11} -x_6x_9 ,    & f_4= x_2x_6 -x_2x_9 +x_1x_5, \cr
f_5 = x_8 -x_2 -a_2x_7  -x_5x_6,     & f_6 = x_{11} -x_2 -a_1x_{12}  -x_9x_{10} -x_9^2
\end{array}
$$
is $W$-homogeneous and satisfies $I \cap P_0 = \{0\}$. 
Therefore $R=P/I$ is a positive $P_0$-algebra.
The two points
\begin{align*}
\Pi_1 &\;=\; (2,\; 1,\;  \tfrac{1}{2},\;\;\,3,\;  5,\;  1,\; -6,\;  1,\;  3,\;  0,\;  0,\; 3,\; -3,\; -3)\cr
\Pi_2 &\;=\; (1,\; 4,\;   7,\;           -9,\;  2,\;  2,\;\;\;0,\;  0,\;  3,\;  3,\;  0,\; 0,\;\;\;3,\;\;12)
\end{align*}
belong to $\Spec(R)$. We connect them by a sequence of irreducible curves.

{\it Step 1:}\/ The point~$\Pi_1$ lies in the same fiber of~$\Theta$ as $\Gamma_0^{(1)} = 
(2,\, 1,\, 0,\dots,\, 0)$. The ideal $I_{\Gamma^{(1)}}$ of this fiber is generated by the
polynomials $\tilde{f}_1^{\mathstrut},\dots,\tilde{f}_6 \in \QQ[x_1,\dots,x_{12}]$,
where $\tilde{f}_i$ is obtained from~$f_i$ by letting $a_1 \mapsto 2$ and $a_2 \mapsto 1$.

To construct an irreducible curve connecting~$\Pi_1$ and~$\Gamma_0^{(1)}$ we start with
the $\QQ$-algebra homomorphism $\Psi\colon\QQ[X] \longrightarrow \QQ[t]$ such that 
$\Psi(x_1) = \tfrac{1}{2}\, t^2$, $\Psi(x_2) = 3t^2$, $\Psi(x_3)= 5 t^3$, $\Psi(x_4) = t^3$,
$\Psi(x_5) = -6t$, $\Psi(x_6)=t$, $\Psi(x_7)=3t^2$, $\Psi(x_8)=\Psi(x_9)=0$,
$\Psi(x_{10})=3t$, and $\Psi(x_{11})= \Psi(x_{12}) = -3t^2$. 
Notice that~$\Psi$ is homogeneous of degree zero and surjective, as $t=\Psi(x_6)$.

Moreover, for $i=1,\dots,6$, we have $\Psi(\tilde{f}_i) = 0$. Therefore~$\Psi$ induces a
homogeneous $\QQ$-algebra homomorphism $\overline{\Psi}\colon\QQ[X]/I_{\Gamma^{(1)}} \longrightarrow \QQ[t]$.
Thus we get a morphism $\psi=\Spec(\overline{\Psi})\colon\AA^1_\QQ \longrightarrow F_{\Gamma^{(1)}}$ 
whose image is an irreducible
curve which contains $\Pi_1$, defined by $\overline{\Psi}^{-1}(\langle t-1\rangle)$, and~$\Gamma_0^{(1)}$,
defined by $\overline{\Psi}^{-1}(\langle t\rangle)$.

{\it Step 2:}\/ Let $\Gamma_0^{(2)} = (1,\, 4,\, 0, \dots,\, 0)$ be the zero point in the
fiber of~$\Pi_2$. To connect~$\Gamma_0^{(1)}$ and $\Gamma_0^{(2)}$, we use the
line joining $\Gamma^{(1)} = (2,1)$ and $\Gamma^{(2)} = (1,4)$ in $\AA^1_\QQ$
and embed it into $\Spec(R)$, i.e., we use the morphism $\lambda\colon\AA^2_\QQ \longrightarrow \Spec(R)$
given by $\lambda(t) = (2-t,\, 1+3t,\, 0,\dots,\, 0)$.

{\it Step 3:}\/ It remains to connect $\Gamma_0^{(2)}$ with $\Pi_2$ by an irreducible
curve in~$F_{\Gamma^{(2)}}$. We proceed as in Step~1, except that this time we use the 
$\QQ$-algebra homomorphism $\Psi'\colon\QQ[X] \longrightarrow \QQ[t]$ such that 
$\Psi'(x_1) = 7 t^2$, $\Psi'(x_2) = -9 t^2$, $\Psi'(x_3)= \Psi'(x_4) = 2 t^3$,
$\Psi'(x_5) = \Psi'(x_6)= 0$, $\Psi'(x_7) = \Psi'(x_8) = 3t^2$, $\Psi'(x_9)=\Psi'(x_{10})=0$,
$\Psi'(x_{11}) = 3t^2$, and $\Psi'(x_{12}) = 12 t^2$. 
Altogether, these  three irreducible curves connect $\Pi_1$ to~$\Pi_2$.
\end{example}

\bigbreak
%%%%%%%%%%%%%%%%%%%%%%%%%%%%%%%%%%%%%%%%%%%%%%%%%%%%%%%
%
% Section 4: Singularities
%
%%%%%%%%%%%%%%%%%%%%%%%%%%%%%%%%%%%%%%%%%%%%%%%%%%%%%%%

\section{Singularities}
\label{sec-Singularities}

The following invariant will play a central role in this section.

\begin{definition}\label{def-embcodim}
Let $(R,\m)$ be a local $K$-algebra with residue field $K=R/\m$.
\begin{enumerate}
\item[(a)] The $K$-vector space $\Cot_\m(R) = \m/\m^2$ is called the {\bf cotangent space} of~$R$.
Its dimension $\edim(R) = \dim_K(\m/\m^2)$ is called the {\bf embedding dimension} of~$R$. 

\item[(b)] The number $\ecod(R) = \edim(R) - \dim(R)$ is called the {\bf embedding codimension}
of~$R$. 
\end{enumerate}
\end{definition}

It is well known that $\ecod(R) \ge 0$ and that $\ecod(R)=0$ if and only if~$R$ is a regular local ring.

Returning to the setting of the preceding sections, we let $P_0=K[a_1,\dots, a_m]$ be trivially
graded, let $P=P_0[x_1,\dots, x_s]$ be positively graded as a $P_0$-algebra, and let $R=P/I$ be a positive
$P_0$-algebra. For the  morphism $\Theta \colon \Spec(R) \longrightarrow \AA^m_K$,
we define the following sets of singular points.

\begin{definition}\label{def-SingSets}
Let $R=P/I$ be a positive $P_0$-algebra, let $\Theta \colon \Spec(R) \longrightarrow \AA^m_K$
be its canonical morphism, and let $Z_R \subseteq \Spec(R)$ be its zero section.
\begin{enumerate}
\item[(a)] The set 
$$
\Sing_0(\Theta) = \{ \Gamma\in K^m \mid 
\ecod( R_{\Gamma_0})>0 \} = \{ \Gamma\in K^m \mid \Gamma_0 \in \Sing(R) \}  
$$ 
is called the set of {\bf zero section singularities} of~$\Theta$.

\item[(b)] The set 
$$
\Sing_v(\Theta) = \{ \Gamma \in K^m \mid 
\ecod(R_{F_\Gamma}) >0\} = \{ \Gamma \in K^m \mid (0,\dots,0) \in \Sing(F_\Gamma) \} 
$$ 
is called the set of {\bf fiber vertex singularities} of~$\Theta$.

\item[(c)] The set 
$$
\Sing_s(\Theta) = \{ \Gamma \in K^m \mid \{ (0,\dots,0) \} \subsetneq \Sing(F_\Gamma) \}
$$ 
is called the {\bf singular fiber locus} of~$\Theta$.

\end{enumerate}
\end{definition}

The following proposition extends  results proved in~\cite[Sect.~5]{KR5}.
It provides important inequalities which allow us to study the above sets of singularities.

\begin{proposition}{\bf (The Fundamental Inequalities)}\label{prop-fundamental}\\
Let $P_0=K[a_1,\dots, a_m]$ be trivially graded, let $P=P_0[x_1,\dots, x_s]$ be positively graded
as a $P_0$-algebra, let $R=P/I$ be a positive $P_0$-algebra, and let $\Gamma=(\gamma_1, \dots, \gamma_m)$ be a point in~$K^m$.  
\begin{enumerate}
\item[(a)] We have 
$$
\begin{array}{ll}
d_{\Gamma_0} - m   \;\le\;   d_{F_\Gamma} \hspace{-8pt} & {\overset{(1)}{=}}\ 
\dim_K(\Cot_{\langle X\rangle / I_\Gamma}( K[X]/I_\Gamma ) ) -\ecod(R_{F_\Gamma})\cr
& \overset{(2)}{=} \ {\dim_K( \Cot_{\m_{\Gamma_0}}(R)) -m - \ecod(R_{F_\Gamma})}^{\mathstrut}\cr
&\overset{(3)}{=} \ d_{\Gamma_0}      + \ecod(R_{\Gamma_0}) -m  - \ecod(R_{F_\Gamma})
\end{array}
$$

\item[(b)] We have $0\le \ecod(R_{F_\Gamma}) \le \ecod(R_{\Gamma_0})$.
\end{enumerate}
\end{proposition}

\begin{proof}
First we prove~(a).
The inequality follows from the fact that $R_{F_\Gamma}$ is obtained as the 
residue class ring of~$R_{\Gamma_0}$ modulo the ideal $\langle a_1-\gamma_1, 
\dots, a_m-\gamma_m \rangle$. 
Equality~(1) follows immediately from the definition of $\ecod(R_{F_\Gamma})$.
Equality~(2) follows from~\cite[Theorem ~5.9.b]{KR5}.
Equality~(3)  follows directly from the definition of $\ecod(R_{\Gamma_0})$.

Finally, claim~(b) follows from $d_{\Gamma_0} - m   \le d_{\Gamma_0} - m  
+ \ecod(R_{\Gamma_0}) - \ecod(R_{F_\Gamma})$ proved above.
\end{proof}

This proposition allows us to compare the sets $\Sing_0(\Theta)$ and $\Sing_v(\Theta)$ as follows.

\begin{corollary}\label{cor-singsing}
Let $R=P/I$ be a positive $P_0$-algebra, and let
$\Theta \colon \Spec(R) \longrightarrow \AA^m_K$
be its canonical morphism. Then
$\Sing_v(\Theta)\subseteq\Sing_0(\Theta)$.
\end{corollary}

\begin{proof}
The claim follows from Proposition~\ref{prop-fundamental}.b.
\end{proof}

Next, we characterize the regularity of the local ring at a point in the zero section
of~$\Theta$.

\begin{corollary}{\bf (Regularity of $R_{\Gamma_0}$)}\label{cor-regR}\\
Let $R=P/I$ be a positive $P_0$-algebra as above, and let $\Gamma=(\gamma_1, \dots, \gamma_m) \in K^m$.
Then the following conditions are equivalent.
\begin{enumerate}
\item[(a)] The local ring $R_{\Gamma_0}$ of the zero point over~$\Gamma$ is regular.

\item[(b)]  The local ring $R_{F_\Gamma}$ of the fiber at its origin is regular, 
and we have the equality $d_{F_\Gamma} = d_{\Gamma_0} - m$.
\end{enumerate}

In addition, if the ring $R_{F_\Gamma}$ is a regular local ring, 
then the fiber $F_\Gamma$ is isomorphic to an affine space.
\end{corollary}

\begin{proof} 
First we prove $(a)\Rightarrow(b)$.
The regularity of $R_{F_\Gamma}$ follows from Proposition~\ref{prop-fundamental}.b,
and the equality follows from 
$d_{F_\Gamma} = d_{\Gamma_0}    -m  + \ecod(R_{\Gamma_0}) - \ecod(R_{F_\Gamma})$
proved in Proposition~\ref{prop-fundamental}.a.

To prove  $(b)\Rightarrow(a)$ we notice that,  using equality (3) of Proposition~\ref{prop-fundamental}.a, 
 the regularity of $R_{F_\Gamma}$ and  $d_{F_\Gamma} = d_{\Gamma_0} -m$
imply  $\ecod(R_{\Gamma_0})=0$. 

The additional claim is proved in~\cite[Theorem 5.9.d]{KR5}.
\end{proof}

The regularity of the local ring of a fiber of~$\Theta$ at its origin can be characterized as follows.

\begin{corollary}{\bf (Regularity of $R_{F_\Gamma}$)}\label{cor-regF}\\
Let $R=P/I$ be a positive $P_0$-algebra as above, and let $\Gamma=(\gamma_1, \dots, \gamma_m) \in K^m$. 
Then the following conditions are equivalent.
\begin{enumerate}
\item[(a)] The local ring $R_{F_\Gamma}$ of the fiber at its origin is regular.

\item[(b)] We have $ d_{F_\Gamma}  =  \dim_K( \Cot_{\m_{\Gamma_0}}(R)) -m =  d_{\Gamma_0} -m +   \ecod(R_{\Gamma_0})$.
\end{enumerate}
\end{corollary}

\begin{proof}
The equivalence  follows from 
equalities (2) and (3)  in Proposition~\ref{prop-fundamental}.
\end{proof}

\begin{remark}\label{rem-affine}
The ring $R_{F_\Gamma} = (K[X]/I_\Gamma)_{\langle X\rangle /I_\Gamma}$ 
is the localization of the positively graded ring at its irrelevant maximal ideal. 
Thus we obtain the equalities $\dim(F_\Gamma) = \dim(K[X]/I_\Gamma) = \dim(R_{F_\Gamma}) = d_{F_\Gamma}$.
\end{remark}

The following example illustrates some aspects of the preceding results.

\begin{example}\label{ex-baddim}
Let $P_0=\QQ[a, b]$, let $P=P_0[x, y]$ be graded by $W =(1,1)$, let $I = \langle ax +by\rangle$, 
let $R=P/I$, and let $\Gamma = (0,0)$. Then the map $\theta \colon  P_0 \To R$ is injective 
and the coordinate ring of the fiber  over~$\Gamma$ is $\QQ[x, y]$. 

Here $R_{F_\Gamma} = \QQ[x, y]_{\langle x, y \rangle}$ is regular, but the local ring 
$R_{\Gamma_0} =  R_{\langle a, b, x, y\rangle}$ is singular.
Note that $d_{F_\Gamma} = 2$, $\ecod(R_{\Gamma_0}) =1$, $d_{\Gamma_0} =3$, and $m =2$. 
Hence we have $ d_{F_\Gamma}  =    \ecod(R_{\Gamma_0}) +d_{\Gamma_0} -m$ as prescribed by 
Corollary~\ref{cor-regF}.b.

Notice that in this case the inclusion proved in Corollary~\ref{cor-singsing} is strict.
Finally, note that $R$ is an integral domain, and thus $\Spec(R)$ is irreducible.
\end{example}

\bigbreak
%%%%%%%%%%%%%%%%%%%%%%%%%%%%%%%%%%%%%%%%%%%%%%%%%%%%%%%%%
%
% Section 5. P_0-Linear Coefficient Matrices
%
%%%%%%%%%%%%%%%%%%%%%%%%%%%%%%%%%%%%%%%%%%%%%%%%%%%%%%%%%

\section{ \texorpdfstring{$P_0$}{P0}-Linear Coefficient Matrices}

\label{sec-AlinCoeffMat}

Continuing in the setting of the preceding sections, let $P_0=K[a_1,\dots,a_m]$ be trivially
graded, let $P=P_0[x_1,\dots,x_s]$ be positively graded, and let $R=P/I$ be a positive $P_0$-algebra.
Recall that, for every $\Gamma=(\gamma_1,\dots,\gamma_m)\in K^m$, we defined the ideal
$I_\Gamma$ as the image of~$I$ under the epimorphism $\epsilon_\Gamma\colon P\longrightarrow K[X]$
given by $\epsilon_\Gamma(a_i)=\gamma_i$ for $i=1,\dots,m$ and $\epsilon_\Gamma(x_j)=x_j$
for $j=1,\dots,s$.

Moreover, for a point
$\Pi=(\gamma_1,\dots,\gamma_m,\pi_1,\dots,\pi_s)\in K^n$, the linear part of a polynomial
$f\in P$ with respect to the maximal ideal
$\M_\Pi=\langle a_1-\gamma_1,\dots,a_m-\gamma_m,x_1-\pi_1,\dots,x_s-\pi_s\rangle$
is defined in~\cite[Sect.~1]{KLR1} as
$\Lin_{\M_\Pi}(f)=\ell(a_1-\gamma_1,\dots,a_m-\gamma_m,x_1-\pi_1,\dots,x_s-\pi_s)$,
where~$\ell$ is the homogeneous component of standard degree~1 of
$f(a_1+\gamma_1,\dots,a_m+\gamma_m,x_1+\pi_1,\dots,x_s+\pi_s)$.

\medskip
The following definition generalizes~\cite[Def.~1.6]{KLR1}.

\begin{definition}\label{def-linA0}
Let $f\in P$, and let $I$ be an ideal in~$P$.
\begin{enumerate}
\item[(a)] Write $f = f_0 + f_1 x_1 + \cdots + f_s x_s + g$ with $f_0,\dots,f_s \in P_0$
and $g\in \langle X \rangle^2$. Then $\Lin_{P_0}(f) = f_1 x_1 + \cdots + f_s x_s$ is called
the {\bf $P_0$-linear part} of~$f$.

\item[(b)] The $P_0$-module generated by the $P_0$-linear parts of the polynomials
in~$I$ is called the {\bf $P_0$-linear part} of~$I$ and denoted by $\Lin_{P_0}(I)$. In other words,
we let $\Lin_{P_0}(I) = \langle \Lin_{P_0}(f) \mid f\in I \rangle_{P_0}$.
\end{enumerate}
\end{definition}

The following example illustrates this definition.

\begin{example}\label{ex-Lin0}
Let $P_0=\QQ[a]$, let $P=P_0[x_1,x_2, x_3]$, and let $W=(1,1,1)$.
Then we have $\Lin_{P_0}(a -ax_1 + x_2 + a^2 x_2 +a^2x_1^2 - x_3^2) = -ax_1 + (1+a^2)\, x_2$.
\end{example}

The following properties were shown in~\cite[Sect.~1]{KLR1} and~\cite[Theorem~5.9]{KR5}.

\begin{remark}\label{rem-Lin0Props}
Let $R=P/I$ be a positive $P_0$-algebra,
let $\Gamma = (\gamma_1,\dots,\gamma_m) \in K^m$, and let 
$\Gamma_0 = (\gamma_1, \dots, \gamma_m,\, 0, \dots,0)$ be the 
zero point over~$\Gamma$.  Then the following claims hold true.
\begin{enumerate}
\item[(a)] By~\cite[Theorem~5.9.a]{KR5}, we have 
$\Lin_{\langle X\rangle}(I_\Gamma) = \Lin_{\M_{\Gamma_0}}(I)$.

\item[(b)] It follows from the definition that we have the equalities
$\Lin_{\langle X\rangle}(I_\Gamma) = \Lin_{\M_{\Gamma_0}}(I) = \epsilon_\Gamma(\Lin_{P_0}(I))$.

\item[(c)] Given a set of generators $\{g_1,\dots, g_t\}$ of~$I$, we have
$$
\Lin_{\langle X\rangle}(I_\Gamma) = \langle \Lin(\epsilon_\Gamma(g_1)), \dots, \Lin(\epsilon_\Gamma(g_t))\rangle_K
$$
by~\cite[Proposition~1.6]{KLR1}.

\end{enumerate}
\end{remark}

The following proposition generalizes  Remark~\ref{rem-Lin0Props}.c by showing 
that $\Lin_{P_0}(I)$ can be computed from any system of generators of~$I$.

\begin{proposition}\label{prop-indepOfGen}
Let $R=P/I$ be a positive $P_0$-algebra, let $G=\{g_1,\dots,g_t\}$
be a system of generators of~$I$, and let $\Gamma=(\gamma_1,\dots,\gamma_m)\in K^m$.
Then the following claims hold true.
\begin{enumerate}
\item[(a)] 
$\Lin_{P_0}(I)=\langle X\rangle_{P_0}\cap (I+\langle X\rangle^2)$.

\item[(b)] 
$\Lin_{P_0}(I)=
\langle \Lin_{P_0}(g_1),\dots,\Lin_{P_0}(g_t)\rangle_{P_0}$.

\item[(c)] 
The map~$\epsilon_\Gamma$ induces an epimorphism of $K$-vector spaces
$$
\Cot_{\m_{\Gamma_0}}(R)\twoheadrightarrow
\langle X\rangle_K/\Lin_{\langle X\rangle}(I_\Gamma)
\cong
\Cot_{\langle X\rangle/I_\Gamma}(K[X]/I_\Gamma).
$$
\end{enumerate}
\end{proposition}

\begin{proof}
To prove~(a), we start with the containment ``$\subseteq$''.
Since~$R$ is a positive $P_0$-algebra, we have $I\subseteq\langle X\rangle$.
Hence, if we let
$\langle X\rangle_{P_0}=P_0x_1+\cdots+P_0x_s$,
then an element $f\in I$ can be written uniquely in the form
$f=\ell+h$ with $\ell\in\langle X\rangle_{P_0}$ and
$h\in\langle X\rangle^2$.
Then $\Lin_{P_0}(f)=\ell\in\langle X\rangle_{P_0}$ and
$\ell=f-h\in I+\langle X\rangle^2$, which proves the first containment.

Conversely, let $\ell\in\langle X\rangle_{P_0}$ be of the form
$\ell=f+h$ with $f\in I$ and $h\in\langle X\rangle^2$.
Then $f=\ell-h$ shows that
$\ell=\Lin_{P_0}(f)\in\Lin_{P_0}(I)$, and this concludes the proof.

Next we show~(b). For $i=1,\dots,t$, we write
$g_i=\Lin_{P_0}(g_i)+h_i$ with
$h_i\in\langle X\rangle^2$.
Using~(a) and the modular law, we get
\begin{align*}
\Lin_{P_0}(I) &\;=\; \langle X \rangle_{P_0} \cap (\langle g_1, \dots, g_t \rangle + \langle X \rangle^2)\\
&\;=\; \langle X \rangle_{P_0} \cap 
(\langle \Lin_{P_0}(g_1), \dots, \Lin_{P_0}(g_t) \rangle_{{P_0}} + \langle X \rangle^2)\\
&\;=\;  \langle \Lin_{P_0}(g_1), \dots, \Lin_{P_0}(g_t) \rangle_{P_0}
+ ( \langle X \rangle_{P_0} \cap \langle X \rangle^2 ) \\
&\;=\; \langle \Lin_{P_0}(g_1), \dots, \Lin_{P_0}(g_t) \rangle_{P_0}
\end{align*}
as was to be shown.

Finally, to prove~(c), we note that~\cite[Proposition~1.8.b]{KLR1}
yields
\[
\Cot_{\m_{\Gamma_0}}(R)
\cong
\M_1/\Lin_{\M_{\Gamma_0}}(I),
\]
where
$\M_1=\langle a_1-\gamma_1,\dots,a_m-\gamma_m,x_1,\dots,x_s\rangle_K$.
Therefore the claim follows from Remark~\ref{rem-Lin0Props}.b.
\end{proof}

To have an input for explicit algorithms, we put the coefficients of the linear parts
of a system of homogeneous generators of~$I$ into a matrix in the following way.

\begin{definition}\label{def-L0Matrix}
Let $R=P/I$ be a positive $P_0$-algebra, and let $G=(g_1,\dots,g_t)$
be a tuple which is a homogeneous system of generators of~$I$.
For $i=1,\dots,t$, we write 
$\Lin_{P_0}(g_i)=g_{i1}x_1+\cdots+g_{i,s}x_s$ with $g_{ij}\in P_0$.
\begin{enumerate}
\item[(a)] The matrix $\cmat_{P_0}(G)=(g_{ij})\in\Mat_{t,s}(P_0)$ is called
the {\bf $P_0$-linear coefficient matrix} of~$G$.

\item[(b)] For $\Gamma=(\gamma_1,\dots,\gamma_m)\in K^m$, let
$G_\Gamma=\{\epsilon_\Gamma(g_1),\dots,\epsilon_\Gamma(g_t)\}$.
Then the matrix $\cmat_K(G_\Gamma)=(\epsilon_\Gamma(g_{ij}))\in\Mat_{t,s}(K)$
is called the {\bf $K$-linear coefficient matrix} of 
$G_\Gamma$ at the point~$\Gamma$.

\end{enumerate}
\end{definition}

For every $\Gamma\in K^m$, the matrix $\cmat_K(G_\Gamma)$ is a
Jacobian matrix of a system of generators of~$I_\Gamma$ at the origin
of the fiber~$F_\Gamma$ and yields a presentation of the cotangent 
space of~$R_{F_\Gamma}$ at this point
(see~\cite[Proposition~1.8]{KLR1}).
Thus the matrix~$\cmat_{P_0}(G)$ may be viewed as a family of matrices 
presenting the cotangent spaces of the fibers of~$\Theta$ at the points
of the zero section.

\smallskip
The generic fiber of~$\Theta$ is defined as follows.

\begin{definition}\label{def-GenFib}
Let $L=K(a_1,\dots,a_m)$ be the field of fractions of~$P_0$.
Then the ideal $I_L=I\,L[X]$ is called the {\bf generic fiber ideal} of~$\Theta$, and
$R_L=L[X]/I_L$ is the coordinate ring of the {\bf generic fiber} of~$\Theta$.
\end{definition}

The $P_0$-linear coefficient matrix has the following basic properties.

\begin{remark}\label{rem-blockmat}
Let $R=P/I$ be a positive $P_0$-algebra, and let $G=(g_1,\dots,g_t)$
be a homogeneous system of generators of~$I$ such that
$\deg_W(g_1)\le\cdots\le\deg_W(g_t)$.
Denote the distinct $W$-degrees of the elements of~$G$ by
$1\le d_1<\cdots<d_\nu$.
\begin{enumerate}
\item[(a)] The $P_0$-linear coefficient matrix $\cmat_{P_0}(G)$ 
is a block diagonal matrix.
The block corresponding to the $W$-degree $d_j$ has
$\# \{ g_i\in G\mid \deg_W(g_i)=d_j\}$ rows and
$\# \{x_\ell\mid \deg_W(x_\ell)=d_j\}$ columns, for
$j=1,\dots,\nu$.

\item[(b)] The local ring at the origin of the generic fiber is
$(R_L)_{\langle X\rangle R_L}$.
The cotangent space of this ring is
$$
\Cot_{\langle X\rangle R_L}((R_L)_{\langle X\rangle R_L})
\;\cong\;
\langle X\rangle_L \,/\,
\langle \Lin_{P_0}(g_1),\dots,\Lin_{P_0}(g_t)\rangle_L .
$$
The matrix $\cmat_{P_0}(G)$ is the coefficient matrix of
$(\Lin_{P_0}(g_1),\dots,\Lin_{P_0}(g_t))$
with respect to the basis $(x_1,\dots,x_s)$ of $\langle X\rangle_L$.

\item[(c)] Let $d_L=\dim(R_L)$ be the dimension of the generic fiber.
Then
$$
\dim_L\bigl(
\Cot_{\langle X\rangle R_L}((R_L)_{\langle X\rangle R_L})
\bigr)
=
s-\rk(\cmat_{P_0}(G))
\ge d_L .
$$
Hence the origin of the generic fiber is a smooth point if and only if
$\rk(\cmat_{P_0}(G))=s-d_L$.

\end{enumerate}
\end{remark}

\medskip

In the following example, various aspects of the preceding remark come into play.

\begin{example}\label{ex-LinMatrix}
Let $P_0=\QQ[a]$ be trivially graded, let $P=P_0[x,y,z]$ be graded by
$W=(1,2,2)$, and let $R=P/I$ be the positive $P_0$-algebra defined by
$I=\langle g\rangle$, where $g=y-ay+a^2z-x^2$.

Then the $P_0$-linear part of~$I$ is
$\Lin_{P_0}(I)=\langle \Lin_{P_0}(g)\rangle_{P_0}
=\langle (1-a)y+a^2z\rangle_{P_0}$.
For the tuple $G=(g)$, the $P_0$-linear coefficient matrix is
$\cmat_{P_0}(G)=(0\;\;1-a\;\;a^2)$.

Finally, letting $L=\QQ(a)$, the coordinate ring of the generic fiber is
$R_L=L[x,y,z]/\langle y+\frac{a^2}{1-a}z-\frac{1}{1-a}x^2\rangle$,
and hence we have an isomorphism $R_L\cong L[x,z]$.
\end{example}

\bigbreak
%%%%%%%%%%%%%%%%%%%%%%%%%%%%%%%%%%%%%%%%%%%%%%%%%%%%%%
%
% Section 6: Computing the Singularities of~$\Theta$
%
%%%%%%%%%%%%%%%%%%%%%%%%%%%%%%%%%%%%%%%%%%%%%%%%%%%%%%

\section{Computing the Singularities of \texorpdfstring{$\Theta$}{Theta}}
\label{sec-CompSing}

In this section we use the fundamental inequalities (see Proposition~\ref{prop-fundamental})
and the $P_0$-linear coefficient matrix (see Def.~\ref{def-L0Matrix})
to compute the singular loci associated to~$\Theta$.
As before, we let $P_0=K[a_1,\dots,a_m]$ be trivially graded, let
$P=P_0[x_1,\dots,x_s]$ be positively graded, and let
$R=P/I$ be a positive $P_0$-algebra.
Moreover, let $\Theta\colon\Spec(R) \longrightarrow \AA^m_K$ be the canonical morphism,
and let~$G$ be a tuple of $W$-homogeneous polynomials which generate~$I$.
Recall that the set of zero section singularities of~$\Theta$ is defined as
$$
\Sing_0(\Theta)=\{\Gamma\in K^m\mid \ecod(R_{\Gamma_0})>0\}
$$

In this setting, the $P_0$-linear coefficient matrix
$\cmat_{P_0}(G)$ can be used to characterize
$\Sing_0(\Theta)$ as follows.

\begin{theorem}{\bf (Characterization of Zero Section Singularities)}
\label{thm-CharSing0}\\
Let $R=P/I$ be a positive $P_0$-algebra of dimension~$d$, let
$\Theta\colon\Spec(R)\longrightarrow\AA^m_K$ be its canonical morphism,
let $G=(g_1,\dots,g_t)$ be a tuple of $W$-homogeneous polynomials
which generates~$I$, let $\cmat_{P_0}(G)$ be the
$P_0$-linear coefficient matrix of~$G$, and let $J_{n-d}$ be the ideal of~$P_0$
generated by all minors of size $n-d$ of~$\cmat_{P_0}(G)$.
\begin{enumerate}
\item[(a)] For every $\Gamma\in K^m$, we have
$$
\rk(\cmat_K(G_\Gamma))
=
n-\dim_K(\Cot_{\m_{\Gamma_0}}(R))
\le n-d_{\Gamma_0}.
$$

\item[(b)] We have
$ \Sing_0(\Theta) =
\{\Gamma\in K^m\mid
\rk(\cmat_K(G_\Gamma))<n-d_{\Gamma_0}\}$.

\item[(c)] If $\rk(\cmat_{P_0}(G))<n-d$, then
$\Sing_0(\Theta)=\AA^m_K$,
i.e., the zero section $Z_R$ is contained in the singular locus of~$\Spec(R)$.

\item[(d)] If $\rk(\cmat_{P_0}(G))=n-d$, then
$\mathcal Z(J_{n-d})\subseteq\Sing_0(\Theta)$.

\item[(e)] Assume that $\Spec(R)$ is equidimensional. Then
$\rk(\cmat_{P_0}(G))\le n-d$ and
$\Sing_0(\Theta)=\mathcal Z(J_{n-d})$ i.e., the set of zeros of~$J_{n-d}$.
\end{enumerate}
\end{theorem}

\begin{proof}
First we prove the equality stated  in~(a). By~\cite[Proposition~1.5.b]{KLR1}, we have
the equality
$n-\dim_K(\Cot_{\m_{\Gamma_0}}(R))
=\dim_K(\Lin_{\M_{\Gamma_0}}(I))$,
and Remark~\ref{rem-Lin0Props} implies
\[
\Lin_{\M_{\Gamma_0}}(I)
=
\Lin_{\langle X\rangle}(I_\Gamma)
=
\langle
\Lin(\epsilon_\Gamma(g_1)),\dots,
\Lin(\epsilon_\Gamma(g_t))
\rangle_K .
\]
Since the matrix $\cmat_K(G_\Gamma)$ is the coefficient matrix of the latter
$K$-vector space with respect to the basis
$\{x_1,\dots,x_s\}$, its rank equals the dimension of that vector space,
and the claimed equality follows. The inequality is a consequence of
$\dim(R_{\Gamma_0})\leq
\dim_K(\Cot_{\m_{\Gamma_0}}(R))$ (see for instance~\cite[V.4.5]{Kun}).

Claim~(b) follows from~(a), since the point $\Gamma_0$ is a singular point of
$\Spec(R)$ if and only if
$d_{\Gamma_0}<\dim_K(\Cot_{\m_{\Gamma_0}}(R))$.
This is equivalent to
$d_{\Gamma_0}<n-\rk(\cmat_K(G_\Gamma))$,
i.e., to
$\rk(\cmat_K(G_\Gamma))<n-d_{\Gamma_0}$.

To prove~(c), we may assume that $K$ is algebraically closed.
For every $\Gamma\in K^m$, we have
\[
\rk(\cmat_K(G_\Gamma))
\leq \rk(\cmat_{P_0}(G)) < n-d
\leq n-d_{\Gamma_0},
\]
and this implies the claim by~(b).

Now we show~(d). Again we may assume that $K$ is algebraically closed.
For every $\Gamma\in\mathcal Z(J_{n-d})$, all minors of size $n-d$ of
$\cmat_K(G_\Gamma)$ vanish. Hence
\[
\rk(\cmat_K(G_\Gamma))<n-d\leq n-d_{\Gamma_0},
\]
and therefore $\Gamma\in\Sing_0(\Theta)$.

Finally, we prove~(e). To show the first claim, we may assume that
$K$ is algebraically closed.
Suppose that
$\rk(\cmat_{P_0}(G))>n-d$.
Then there exists a point $\Gamma\in K^m$ such that
$\rk(\cmat_K(G_\Gamma))>n-d$.
By~(a), we have the inequality $\rk(\cmat_K(G_\Gamma)) \le n-d_{\Gamma_0}$
which implies $n -d < n - d_{\Gamma_0}$ and hence $d_{\Gamma_0}<d$.
Consequently, every irreducible component of $\Spec(R)$ containing $\Gamma_0$
has dimension~$<d$, in contradiction to the hypothesis. 

It remains to show the second claim in~(e).
Since $\Spec(R)$ is equidimensional of dimension $d$, we have
\[
\Sing_0(\Theta)
=
\{\Gamma\in K^m\mid
\rk(\cmat_K(G_\Gamma))<n-d\}
\]
by~(b), and this set is defined as the zero locus of the ideal $J_{n-d}$.
\end{proof}

\medskip
If $\Spec(R)$ is equidimensional, part~(e) of this theorem allows us to compute
$\Sing_0(\Theta)$. The following corollary extends this result by providing an
algorithm for computing $\Sing_0(\Theta)$ in general.

\begin{corollary}{\bf (Computing the Set of Zero Section Singularities)}
\label{cor-CompSing0}$\mathstrut$\\
Let $K$ be a perfect field, and let $R=P/I$ be a positive $P_0$-algebra as above.
Then the following sequence of instructions defines an algorithm which computes
the vanishing ideal of $\Sing_0(\Theta)$ in~$P_0$.
\begin{enumerate}
\item[(1)] Compute the support~$\mathfrak{n}$ of the nilradical of~$R$ using
$\mathfrak{n}=\Ann_P(\Rad(I)/I)$.

\item[(2)] Compute the homogeneous prime ideals
$\mathfrak{p}_1,\dots,\mathfrak{p}_\ell$ in~$P$
which represent the minimal prime ideals of~$I$.

\item[(3)] For $i=1,\dots,\ell$, compute
$d_i=\dim(R/\mathfrak{p}_i)$. Then use part~(b) of the theorem
to compute an ideal $\mathfrak{a}_i$ in~$P_0$ which represents the
contribution to $\Sing_0(\Theta)$ coming from the component
$\mathcal{Z}(\mathfrak{p}_i)$.

\item[(4)] Compute the intersection
$\mathfrak{q}=\mathfrak{n}\cap
\bigcap_{i\ne j}(\mathfrak{p}_i+\mathfrak{p}_j)$.

\item[(5)] Set $x_1=\cdots=x_s=0$ in the generators of~$\mathfrak{q}$,
and let $\mathfrak{b}$ be the resulting ideal in~$P_0$.

\item[(6)] Compute
$\mathfrak{a}_1\cap\cdots\cap\mathfrak{a}_\ell\cap\mathfrak{b}$
and return the resulting ideal.
\end{enumerate}
\end{corollary}

\begin{proof}
By definition of $\Sing_0(\Theta)$, we have to intersect $\Sing(R)$ with the zero section
and then identify the resulting set with a subset of~$\AA^m_K$. 
The set $\Sing(R)$ is the union of the support of the nilradical,
as computed in step~(1), the pairwise intersections of the irreducible components
of $\Spec(R)$, as computed in step~(4), and the singular loci of the individual
irreducible components.
For the latter sets, we use part~(b) of the theorem to compute the corresponding
ideals in~$P_0$ directly in step~(3). Finally, we compute the intersection of the
support of the nilradical and the pairwise intersections of the irreducible components
with the zero section by setting $x_1=\cdots=x_s=0$ in step~(5).
The union of the various contributions to $\Sing_0(\Theta)$ is then computed in the
final step~(6).
\end{proof}

Clearly, additional information about~$R$ may allow us to simplify this algorithm.
For instance, if we know that $R$ is reduced, we can skip step~(1).
The following examples illustrate various applications of the theorem.

\begin{example}\label{ex-cor(a)}
Let $P_0=\QQ[a, b]$, let $P = P_0[x, y]$ be graded by $W=(1, 1)$, 
let $I =\langle ax,\, by^2 \rangle$, and let $R=P/I$.
Then $\Spec(R)$ is equidimensional and has four irreducible
components of dimension $d=2$.

The $P_0$-linear coefficient matrix of $G = (ax,\, by^2)$ is
$\cmat_{P_0}(G)=
\left(\begin{smallmatrix}
a & 0\\
0 & 0
\end{smallmatrix}\right)$.
Thus, we have $\rk(\cmat_{P_0}(G)) = 1 < n - d$, and Theorem~\ref{thm-CharSing0}.c
implies $\Sing_0(\Theta) = \AA^2_\QQ$. (This example is revisited in 
Example~\ref{ex-cor(a)-continued}.)
\end{example}

\begin{example}\label{ex-thm-c}
Let $P_0=\QQ[a]$, let $P = P_0[x, y, z]$ be graded by $W=(2, 2, 1)$,
let $I = \langle ax +z^2,, ay+z^2 \rangle
= \langle a,, z^2\rangle \cap \langle x-y,, z^2\rangle$, and let $R=P/I$.
Then $\Spec(R)$ is equidimensional and has two irreducible components
of dimension $d=2$. The fibers of the
morphism $\Theta\colon\Spec(R) \longrightarrow \AA^1_\QQ$
are 1-dimensional for $a\ne 0$ and 2-dimensional for $a=0$.

Here the $P_0$-linear coefficient matrix of $G = (ax+z^2,, ay+z^2)$ is
$\cmat_{P_0}(G)=
\left(\begin{smallmatrix}
a & 0 & 0\\
0 & a & 0
\end{smallmatrix}\right)$.
Since $\rk(\cmat_{P_0}(G)) = 2 = n - d$, part~(e) of Theorem~\ref{thm-CharSing0}
shows $\Sing_0(\Theta)= {0}$.
\end{example}

\begin{example}\label{ex-cor-b}
Let $P_0 = \QQ[a]$, let $P = P_0[x, y]$ be graded by $W=(1, 1)$,
let $I =\langle ax, ay^2 \rangle$, and let $R=P/I$.
Then $\Spec(R)$ has two irreducible components: one is the zero section, defined scheme-theoretically
by $Z_R= \mathcal{Z}(\langle x,y^2\rangle)$, and the other one is the
fiber over $\Gamma=(0,0)$, which is a plane $\AA^2_\QQ$.

Here the $P_0$-linear coefficient matrix of $G = (ax,\, ay^2)$ is
$\cmat_{P_0}(G)= (\begin{smallmatrix}  a & 0  \\   0 & 0 \end{smallmatrix})$.
Thus, we have $\rk(\cmat_{P_0}(G))= 1 = n - d$, and Theorem~\ref{thm-CharSing0}.d
implies that $\mathcal{Z}(J_1) = {(0)}$ is contained
in $\Sing_0(\Theta) = \AA^1_\QQ$.

Notice that this is not an example of part~(e) of the theorem, since
$\Spec(R)$ is not equidimensional.
\end{example}

Finally, we include an example where $\rk(\cmat_{P_0}(G)) > n-d$.

\begin{example} \label{rem-bigrk}
Let $P_0=\QQ[a]$, let $P = P_0[x,y]$ be graded by $W=(1, 1)$, let $I =\langle ax, ay \rangle$,
and let $R=P/I$.
Then $\Spec(R)$ has two irreducible components, namely the 1-dimensional
zero section $\mathcal{Z}(x,y)$ and the 2-dimensional fiber $\mathcal{Z}(\langle a\rangle)$
over the origin.

Here the $P_0$-linear coefficient matrix of $G = (ax,, ay)$ is
$\cmat_{P_0}(G)= (\begin{smallmatrix} a & 0 \\ 0 & a \end{smallmatrix})$.
Thus, we have $\rk(\cmat_{P_0}(G)) = 2 > n-d$. Consequently, Theorem~\ref{thm-CharSing0}.e
implies that $R$ is not equidimensional. In this case, we have
$\Sing_0(\Theta) = \mathcal{Z}(J_2) = { (0)}$.
\end{example}

Next, we study the problem of computing the set of fiber vertex singularities which was defined as
$$
\Sing_v(\Theta) \;=\; \{ \Gamma \in K^m \mid \ecod(R_{F_\Gamma}) > 0 \}
$$
Once again, we can use the $P_0$-linear coefficient matrix for this purpose.

\begin{theorem}{\bf (Characterization of Fiber Vertex Singularities) }\label{thm-CharSing_v}\\
Let $R=P/I$ be a positive $P_0$-algebra of dimension~$d$, let $\Theta\colon\Spec(R) \longrightarrow
\AA^m_K$ be its canonical morphism, let $G = (g_1,\dots, g_t)$
be a tuple of $W$-homogeneous polynomials which generate~$I$, and let $\cmat_{P_0}(G)$ be 
the $P_0$-linear coefficient matrix of~$G$.  
\begin{enumerate}
\item[(a)]  For every $\Gamma\in K^m$, we have
$$
\rk(\cmat_K(G_\Gamma)) \;=\; n - m  - \dim_K(\Cot_{\langle X\rangle / I_\Gamma}( K[X]/I_\Gamma ) )\;\le\; 
n - m -d_{F_\Gamma}
$$

\item[(b)] $\Sing_v(\Theta) = \{\Gamma \in K^m \mid  \rk(\cmat_K(G_\Gamma))< n - m - d_{F_\Gamma} \}$
\smallskip

\item[(c)] $\Sing_0(\Theta) \setminus \Sing_v(\Theta) = \{\Gamma \in K^m\mid \rk(\cmat_K(G_\Gamma)) 
= n - m - d_{F_\Gamma}  < n - d_{\Gamma_0} \}$

\end{enumerate}
\end{theorem}

\begin{proof} Let us prove (a). The equality follows from \cite[Proposition~1.5.b]{KLR1}.
The inequality follows from
$\dim( R_{\M_{\Gamma_0}} ) \le \dim_K ( \Cot_{\M_{\Gamma_0}/I}(R) )$ (see
for instance~\cite[V.4.5]{Kun}). 
Claim (b) follows directly from (a). Claim (c) follows from (a) and Theorem~\ref{thm-CharSing0}.a.
\end{proof}

Part~(b) of this theorem can be turned into an algorithm for computing $\Sing_v(\Theta)$. However,
since $\Sing_v(\Theta)$ is in general only a constructible subset of $\AA^m_K$, we have to describe
it in the form $\mathcal{Z}(I_1) \setminus \mathcal{Z}(I_2)$. For the following algorithm, we assume that
the readers are familiar with standard operations for constructible sets such as unions, intersections,
and complements.
Moreover, note that we are using the algorithm for computing a comprehensive Gr\"obner basis
of~$I$ which was introduced in the fundamental paper~\cite{Wei} on this topic. 
There the set of pairs $(G_i,C_i)$ corresponding to a comprehensive Gr\"obner basis
of~$I$ is called a {\bf Gr\"obner system}.

\begin{corollary}{\bf (Computing the Set of Vertex Singularities)}\label{cor-CompSingV}$\mathstrut$\\
Let $K$ be a perfect field and $R=P/I$ a positive $P_0$-algebra as above.
Then the following sequence of instructions defines an algorithm which computes $\Sing_v(\Theta)$.
\begin{enumerate}
\item[(1)] Compute a comprehensive Gr\"obner basis of~$I$, considered as an ideal in $P_0[x_1,\dots,x_s]$
with parameters in~$P_0$ and  indeterminates $\{x_1,\dots, x_s\}$. As a result, get constructible sets $C_i$
as well as Gr\"obner bases $G_i$ for~$I$ over~$C_i$.

\item[(2)] For every~$i$, compute the Krull dimension~$\delta_i$ of $K[x_1,\dots,x_s] / \langle \sigma(G_i)\rangle$
where $\sigma(G_i)$ is the specialized Gr\"obner basis under the conditions defined by~$C_i$.

\item[(3)] For $d=0,1,2,\dots$ compute the constructible sets $S_d = \bigcup_{\{i\mid \delta_i\ge d\}}\, C_i$
until $S_d = \emptyset$.

\item[(4)] For $d=0,1,\dots, s-1$, use the ideals generated by the minors of $\cmat_{P_0}(G)$ 
to compute the closed subsets~$T_d$ of~$\AA^m_K$ defining all $\Gamma\in K^m$ 
such that $\rk(\cmat_K(G_\Gamma)) < s -d$.

\item[(5)] Compute the union of the constructible sets $T_d \cap (S_d \setminus S_{d+1})$
for $d=0,\dots, s-1$. Return the result.

\end{enumerate}
\end{corollary}

\begin{proof}
The sets~$S_d$ are the constructible sets defining all points $\Gamma\in K^m$
such that the coordinate ring $K[x_1,\dots,x_s]/\langle G_\Gamma\rangle$ of the fiber has 
dimension at least~$d$. Thus $S_d\setminus S_{d+1}$ is the constructible set over which the dimension of the
fiber is exactly~$d$. In step~(4) we compute for each~$d$ the closed subset of~$\AA^m_K$
where the inequality of part~$(b)$ of the theorem is satisfied, i.e., the subset of $\Sing_v(\Theta)$
corresponding to a fixed fiber dimension~$d$. Altogether, it follows that step~(5) constructs
the entire set $\Sing_v(\Theta)$. 
\end{proof}

An extension of this algorithm allows us to compute the set $\Sing_s(\Theta)$, as well.

\begin{corollary}{\bf (Computing the Set of Singular Fibers)}\label{cor-CompSingS}$\mathstrut$\\
Let $K$ be a perfect field and $R=P/I$ a positive $P_0$-algebra as above. Assume that the fibers 
of~$\Theta$ are equidimensional.
Then the following sequence of instructions defines an algorithm which computes $\Sing_s(\Theta)$.
\begin{enumerate}
\item[(1)] For $d=0,\dots, s-1$, use Corollary~\ref{cor-CompSingV} to compute
the constructible set $U_d = T_d \cap (S_d \setminus S_{d+1})$.

\item[(2)] For $d\in \{ 0,\dots, s-1\}$ such that $U_d \ne \emptyset$, perform the following two steps.

\item[(3)]
Compute a comprehensive Gr\"obner basis of 
$I + \Jac_{s-d}(I)$, where $\Jac_{s-d}(I)$ is the ideal generated by the
minors of order~$s-d$ of the relative Jacobian matrix of~$I$, considered as an ideal in $P_0[x_1,\dots,x_s]$
with parameters in~$P_0$ and the indeterminates $\{x_1,\dots,x_s\}$.
As a result, get constructible sets $C_{di}$ as well as Gr\"obner bases $G_{di}$
for $I+\Jac_{s-d}(I)$ over~$C_{di}$.
 
\item[(4)] For every~$i$, compute the Krull dimension~$\delta_{di}$ of $K[x_1,\dots,x_s] / \langle \sigma(G_{di})\rangle$
where $\sigma(G_{di})$ is the specialized Gr\"obner basis under the conditions defined by~$C_{di}$.

\item[(5)] Compute the constructible set $\bigcup_{\{d \mid U_d\ne 0\}} \bigcup_{\{i\mid \delta_{di}\ge 1\}}\, C_{di}$
and return it.

\end{enumerate}
\end{corollary}

\begin{proof}
As shown in the proof of Corollary~\ref{cor-CompSingV}, the set $U_d$ is the set of all
points $\Gamma\in K^m$ such that the fiber $F_\Gamma$ of~$\Theta$ has dimension~$d$
and a singularity at the origin. The ideal $I+\Jac_{s-d}(I)$ computes the
singular locus of the fibers $F_\Gamma$, since they are assumed to be equidimensional. 
For all specializations in~$C_{di}$, the corresponding rings $R_{F_\Gamma}$ have a singular
locus of dimension~$\delta_{di}$. 
Notice that, since $I_\Gamma$ is $W$-homogeneous, also the singular locus of~$R_{F_\Gamma}$
is weighted homogeneous. Hence step~(5) correctly computes the constructible set
of all $\Gamma\in K^m$ for which the origin is not the only singularity of~$R_{F_\Gamma}$.
\end{proof}

In this corollary, it is possible to remove the hypothesis that all fibers of~$\Theta$
are equidimensional by adjusting the ideals which compute the singular locus
of the fibers~$F_\Gamma$ to $I+\Jac_{s-d_j}(I)$, where the numbers~$d_j$ are the dimensions
of the components of~$F_\Gamma$. Since this extension makes the algorithm rather hard to
read, we leave it to the interested reader.

Let us return to the setting of Example~\ref{ex-cor(a)} and 
examine $\Sing_v(\Theta)$ and $\Sing_s(\Theta)$.

\begin{example}\label{ex-cor(a)-continued}
As in Example~\ref{ex-cor(a)}, we let $P_0=\QQ[a, b]$, let $P = P_0[x, y]$ be graded by $W=(1, 1)$, 
let $I =\langle ax,  by^2 \rangle$, and let $R=P/I$. In particular, we have $n=4$ and $m=2$.
The spectrum of~$R$ is equidimensional and has four irreducible
components of dimension $d=2$. 

The fibers of the canonical morphism $\Theta\colon\Spec(R) \longrightarrow \AA^2_\QQ$ 
are 0-dimensional over $U = \{ (a,b) \mid a\ne 0,\; b\ne 0 \}$.
They are 1-dimensional for $V_1 \cup V_2$ where $V_1 = \{(a,b) \mid a=0, b\ne 0\}$
and $V_2 = \{(a,b) \mid a\ne 0, b =0\}$. Finally, the fiber over~$\Gamma$ is 2-dimensional for $\Gamma = (0, 0)$.

The $P_0$-linear coefficient matrix of $G = (ax,\, by^2)$ is
$\cmat_{P_0}(G)= (\begin{smallmatrix}a & 0 \\  0 & 0 \end{smallmatrix})$.
Depending on the choice of~$\Gamma$ in~$K^m$, there are four different cases.

\begin{enumerate}
\item[{\it Case 1:}]  $\Gamma \in U$. In this case we have
$\rk(\cmat_K(G_\Gamma)) = 1$ and fiber dimension $d_{F_\Gamma} = 0$.
Since $1= \rk(\cmat_K(G_\Gamma)) < n - m -d_{F_\Gamma} =2$, the ring $R_{F_\Gamma}$
has a singular point at the origin. In fact, we have $R_{F_\Gamma} \cong
K[y] / \langle y^2 \rangle$. As the fiber consists of only one point,
we have $\Gamma\notin \Sing_s(\Theta)$.

\smallskip
\item[{\it Case 2:}] $\Gamma \in V_1$. In this case we have
$\rk(\cmat_K(G_\Gamma)) = 0$ and $d_{F_\Gamma} =1$.
Since $0 = \rk(\cmat_{P_0}(G_\Gamma)) < n - m -d_{F_\Gamma} =1$, 
the ring $R_{F_\Gamma}$ has a singular point at the origin.
In fact, we have  $R_{F_\Gamma} \cong K[x, y]/\langle y^2\rangle$,
and thus $\Gamma\in \Sing_s(\Theta)$.

\smallskip
\item[{\it Case 3:}] $\Gamma \in  V_2$. In this case we have 
$\rk(\cmat_K(G_\Gamma)) = 1$ and $d_{F_\Gamma} =1$.
Since $1= \rk(\cmat_K(G_\Gamma)) = n - m -d_{F_\Gamma} =1$,
the ring $R_{F_\Gamma}$ has a regular point at the origin, and consequently
it is regular. In fact, we have $R_{F_\Gamma} \cong K[y]$.

\smallskip
\item[{\it Case 4:}] $\Gamma =(0,0)$. In this case we have
$\rk(\cmat_{P_0}(G_\Gamma)) = 0$ and $d_{F_\Gamma} =2$.
Since $0 = \rk(\cmat_K(G_\Gamma)) = n - m -d_{F_\Gamma} =0$, 
the ring $R_{F_\Gamma}$ is regular. In fact, we have
$R_{F_\Gamma} \cong K[x, y]$.
\end{enumerate}

Altogether, we find that $\Sing_v(\Theta) = \{(a, b) \in \QQ^2 \mid 
b \ne 0 \}$. As we have seen in Example~\ref{ex-cor(a)} that $\Sing_0(\Theta) = \AA^2_\QQ$,
this example is a case where the inclusion $\Sing_v(\Theta) \subseteq \Sing_0(\Theta)$
shown in Corollary~\ref{cor-singsing} is strict. Moreover, we also find
that $\Sing_s(\Theta) = V_1 = \{ (0,b)\in \QQ^2 \mid b\ne 0\}$, and therefore also
the inclusion $\Sing_s(\Theta) \subseteq \Sing_v(\Theta)$ is strict. 
\end{example}

\bigbreak
%%%%%%%%%%%%%%%%%%%%%%%%%%%%%%%%%%%%%%%%%%%%%%%%%%%%%%%%%%%
%
% Section 7: Application to MaxDeg Border Basis Schemes
%
%%%%%%%%%%%%%%%%%%%%%%%%%%%%%%%%%%%%%%%%%%%%%%%%%%%%%%%%%%%

\section{Application to MaxDeg Border Basis Schemes}
\label{sec-ApplBBS}

In this section we apply the above results and algorithms to 
the coordinate rings of certain border basis schemes. Let us begin by recalling
the necessary definitions. For further details see, for instance,~\cite{KR3}, \cite{KR4},
and~\cite{KR6}.

Let $P=K[x_1,\dots,x_n]$ be a polynomial ring over a field~$K$.
A non-empty subset $\OO =  \{t_1,\dots,t_\mu\}$ of the set of terms
$\TT^n = \{ x_1^{\alpha_1} \cdots x_n^{\alpha_n} \mid \alpha_i \ge 0\}$
is called an {\bf order ideal} if $t\in \OO$ and $t'\mid t$ implies
$t'\in \OO$, i.e., if~$\OO$ is a poset ideal in the set~$\TT^n$
which is partially ordered by divisibility.
Then the set
$$
\partial \OO \;=\; (x_1 \OO \cup \cdots \cup x_n\OO) \setminus \OO \;=\; \{b_1,\dots,b_\nu\}
$$
is called the {\bf border} of~$\OO$.

Next we introduce new indeterminates $c_{ij}$ for $i=1,\dots,\mu$ and $j=1,\dots,\nu$
and define the polynomials $g_j = b_j - \sum_{i=1}^\mu c_{ij} t_i$ in $P[C]$, where
$C=\{c_{ij} \mid i=1,\dots,\mu;\; j=1,\dots,\nu\}$.
The set $G = \{ g_1,\dots,g_\nu\}$ is called the {\bf generic $\OO$-border prebasis}.
The residue classes of the terms in~$\OO$ generate $P[C]/\langle G\rangle$ as a $K[C]$-module.
Thus, for $r=1,\dots,n$, we define the {\bf generic multiplication matrices}
$\mathcal{A}_r = (a_{ij}^{(r)}) \in \Mat_\mu(K[C])$
whose entries are given by
$$
a_{ij}^{(r)} \;=\; \begin{cases}
\delta_{im} & \hbox{\rm if\ } x_r t_j = t_m, \\
c_{im} & \hbox{\rm if\ } x_r t_j = b_m.
\end{cases}
$$

\begin{definition}{\bf (Border Basis Schemes)}\label{def-BBS}$\mathstrut$\\
In the above setting, let $I(\BO) \subseteq K[C]$ be the ideal generated
by the entries of the commutator matrices $\mathcal{A}_i \mathcal{A}_j - \mathcal{A}_j
\mathcal{A}_i$, where $1\le i < j \le n$. Then the subscheme of $\mathbb{A}^{\mu \nu}$
defined by $I(\BO)$ is called the {\bf $\OO$-border basis scheme} and is denoted by~$\BO$.
Its coordinate ring is denoted by $B_\OO = K[C] / I(\BO)$.
\end{definition}

In general, the rings $B_\OO$ are $\ZZ$-graded by the {\bf total arrow degree}
which is the grading determined by $W$ such that
$\deg_W(c_{ij}) = \deg(b_j)- \deg(t_i)$ for $i=1,\dots,\mu$ and $j=1,\dots,\nu$.
For our purposes, the most interesting situation occurs when this grading is
non-negative, and we give it the following name.

\begin{definition}{\bf (MaxDeg Border Basis Schemes)}\label{def-MaxDeg}$\mathstrut$\\
Let $\OO=\{t_1,\dots,t_\mu\}$ be an order ideal with border $\partial\OO = \{b_1,\dots,b_\nu\}$.
\begin{enumerate}
\item[(a)] The order ideal $\OO$ is called a {\bf MaxDeg order ideal} if $\deg(b_j) \ge \deg(t_i)$
for all $i=1,\dots,\mu$ and all $j=1,\dots,\nu$. In other words, this means that the total arrow grading
on $K[C]$ is non-negative.

\item[(b)] If $\OO$ is a MaxDeg order ideal, then the $\OO$-border basis scheme $\BO$
is called a {\bf MaxDeg border basis scheme}.

\end{enumerate}
\end{definition}

In the following we assume that $\OO$ is a MaxDeg order ideal. Consequently, the ring $B_\OO$
is non-negatively graded. To connect this setting to the previous sections, we denote
the set of indeterminates~$c_{ij}$ of total arrow degree zero~by $C_0$, 
and by~$C_+$ we denote the set of indeterminates~$c_{ij}$ of positive total arrow degree. 
By~\cite[Thm.~5.3]{KR3}, we know that 
$I(\BO) \cap K[C_0]=\{0\}$, and hence $B_\OO = K[C_0][C_+]/I(\BO)$ is a positive $K[C_0]$-algebra.

\medskip
The next remark collects some basic properties of MaxDeg border basis schemes.

\begin{remark}\label{rem-MaxDegBBS}
Let $\OO$ be a MaxDeg order ideal as above.
\begin{enumerate}
\item[(a)] By Proposition~\ref{prop-Rconnected}.c, the $\OO$-border basis scheme is connected.
Let us discuss this observation briefly. We recall from~\cite[Thm. 2.7 and Rmk. 2.8]{KR4} 
that, for every order ideal $\OO \subset \mathbb T^n$, the scheme $\BO$ is a dense open 
subscheme of $\Hilb^\mu(\AA_K^n)$. By a famous result of R.\ Hartshorne, these Hilbert schemes
are always connected. However, it is not known whether their open subschemes $\BO$
are connected in general.

It is also known that $\Hilb^\mu(\AA_K^n)$ is irreducible for any $n$ 
and $\mu\le 7$ (see~\cite{Maz} and the introduction of~\cite{DJNT}).
Thus, all $\OO$-border basis schemes are connected in these cases.

\item[(b)] If $n=2$ and $K$ is a perfect field, i.e., for planar MaxDeg $\OO$-border basis schemes
defined over a perfect field, it was shown in~\cite[Corollary 3.6]{KR6} that~$\BO$ is an affine cell.
More precisely, there exists a homogeneous $K$-algebra isomorphism $B_\OO \cong K[Y]$,
where $Y$ is a tuple of $\dim(\BO)$ indeterminates.
In particular, we have $\Sing_0(\Theta) = \Sing_v(\Theta) = \Sing_s(\Theta) = \emptyset$,
since all fibers of the canonical morphism~$\Theta$ are affine spaces.
\end{enumerate}
\end{remark}

So, to find interesting examples of non-negatively graded rings $B_\OO$,
we start by considering the case $n=3$, i.e., {\bf spatial border basis schemes}.
The first case is $\OO = {1,x,y,z}$, and it
is still relatively straightforward, as the following example shows.

\begin{example}{\bf (The $\{1, x, y, z\}$-Border Basis Scheme)}\label{ex-1xyz}$\mathstrut$\\
In the polynomial ring $P=K[x,y,z]$ we consider the order ideal $\OO= \{ t_1, t_2, t_3, t_4\}$,
where $t_1=1$, $t_2=z$, $t_3=y$, and $t_4=x$. Its border is $\partial\OO = \{ b_1, \dots, b_6\}$,
where $b_1 = z^2$, $b_2= yz$, $b_3=xz$, $b_4=y^2$, $b_5=xy$, and $b_6=x^2$. (For convenience,
we have ordered $\OO$ and $\partial\OO$ increasingly w.r.t.\ {\tt DegRevLex}.)

Here the total arrow degrees of the indeterminates $c_{ij}$ are $\deg_W(c_{ij}) = 1$
if $i\ge 2$ and $\deg_W(c_{ij})=2$ for $i=1$ and $j=1,\dots,6$. Therefore $K[C]$ is positively graded,
and the canonical morphism $\Theta\colon\Spec(B_\OO) \longrightarrow \AA^0_K = \{0\}$
has only one fiber, namely the entire scheme $\BO$.

The point $p=(0,\dots,0) \in \BO$, also called the {\bf monomial point} of~$\BO$, is
singular (see~\cite[Ex. 5.8]{KR6}). Hence we have $\Sing_0(\Theta) = \Sing_v(\Theta) = \{0\}$.
To check that $\Sing_s(\Theta) = \{0\}$, it suffices to verify that the point
$$
(c_{11}, \dots, c_{46}) \;=\; (-1,-1,-1,-1,-1,-1, 2,1,1,0,0,0, 0,1,0,2,1,0, 0,0,1,0,1,2)
$$
is a singular point of~$\BO$, since in this case the only fiber is $\BO$.
\end{example}

The next case is $n=3$ and $\mu=5$. Here we have essentially two different
subcases. In addition to $1,x,y,z$, the order ideal~$\OO$ contains either the square of an indeterminate
or the product of two indeterminates. Both cases are quite subtle and challenging.
We study them carefully in the next two examples.

\goodbreak

\begin{example}{\bf (The $\{1,x,y,z,z^2\}$-Border Basis Scheme})\label{ex-13z2}$\mathstrut$\\
Let $K$ be a perfect field, let $P = K[x,y,z]$, and consider the order ideal
$\OO = \{t_1,\dots,t_5\}$, where $t_1=1$, $t_2=z$, $t_3=y$, $t_4=x$, and $t_5=z^2$. 
Its border is given by $\partial\OO = \{b_1,\dots,b_8\}$, where $b_1=yz$, $b_2=xz$, $b_3=y^2$,
$b_4=xy$, $b_5=x^2$, $b_6=z^3$, $b_7=yz^2$, and $b_8=xz^2$.
Hence we have $\mu=\#(\OO)=5$ and  $\nu=\#(\partial\OO) = 8$ here.

Clearly, $\BO$ is a MaxDeg border basis scheme embedded in $\AA^{40}_K$.
Its coordinate ring $B_\OO = K[c_{11}, c_{12}, \dots, c_{58}] / I(\BO)$
is non-negatively graded by
\begin{align*}
W &\;=\; (2,  2,  2,  2,  2,  3,  3,  3,  1,  1,  1,  1,  1,  2,  2,  2,  1,  1,  1,  1,\\  
&\qquad\; 1,  2,  2,  2,  1,  1,  1,  1,  1,  2,  2,  2,  0,  0,  0,  0,  0,  1,  1,  1)
\end{align*}

We have $ K[C_0] = K[c_{51}, c_{52}, c_{53}, c_{54}, c_{55}]$ and  $B_\OO$
is a positive $K[C_0]$-algebra.
The ideal $I(\BO)$ is generated by~60 $W$-homogeneous polynomials.
According to Remark~\ref{rem-MaxDegBBS}, the scheme  $\BO$ is irreducible. 
Thus it coincides with its principal component, and therefore 
its dimension is $\dim(B_\OO) = 3 \cdot 5 = 15$.

\medskip
{\it (1) A $Z$-Separating Re-Embedding.}
In order to study the canonical morphism $\Theta\colon\BO \longrightarrow \AA^5_K$, its fibers
and its sets of singularities, we can perform a $Z$-separating re-embedding of~$\BO$,
because $I(\BO)\cap K[C_0] = \{0\}$ shows that the resulting isomorphism fixes the base space
$\AA^5_K = \Spec(K[C_0])$ and induces isomorphisms of the fibers and singular loci.

Using~\cite{KR5}, Algorithm 4.2, we calculate the tuple 
$$
Z \;=\; (c_{11},  c_{12},  c_{13},  c_{14},  c_{15}, c_{16},  c_{17},  c_{18},   c_{21},  c_{22},  
c_{23},  c_{24},  c_{25}, c_{27},  c_{28},  c_{37},  c_{38},  c_{47},  c_{48})
$$
of~19 indeterminates which has the property that the ideal $I(\BO)$ is $Z$-separating. 
Hence there exists an epimorphism $\phi: K[C] \longrightarrow K[Y] $, where $Y=C\setminus Z$ 
is a set of 21 indeterminates, such that~$\phi$ induces an isomorphism of $K[C_0]$-algebras 
$\Phi\colon K[C]/I(\BO) \cong K[Y] / (I(\BO)\cap K[Y])$. 

Next, we let $\M=\langle c_{ij}\rangle$ and compute $\dim_K(\Lin_\M(I(\BO))) =19$. 
By~\cite{KLR1}, Corollary 4.2, it follows that $\Phi$ is an optimal re-embedding.  
Moreover, together with $\dim(\BO) = 15$, we obtain that the origin is a singular
point of~$\BO$.

Consequently, we can now apply the isomorphism~$\Phi$
and continue our investigation by examining the positive $K[C_0]$-algebra $R = K[Y]/J$, 
where  $J = I(\BO)\cap K[Y]$. Notice that the set of indeterminates is
\begin{align*}
Y &\;=\; \{ c_{26},  c_{31},  c_{32},  c_{33},  c_{34},  c_{35},  c_{36},  c_{41},  c_{42}, c_{43},  c_{44},\\ 
&\qquad\; c_{45},  c_{46},  c_{51},  c_{52},  c_{53},  c_{54},  c_{55},  c_{56},  c_{57},  c_{58} \}
\end{align*}
and that their degrees are given by 
$$
W' \;=\; (2,  1,  1,  1,  1,  1,  2,  1,  1,  1,  1,  1,  2,  0,  0,  0,  0,  0,  1,  1,  1)
$$
We have $Y=C_0 \cup C_+$ with $C_0 = \{ c_{51},  c_{52},  c_{53},  c_{54},  c_{55}\}$
and 
$$
C_+ \;=\; \{ c_{26},  c_{31},  c_{32},  c_{33},  c_{34},  c_{35},  c_{36},  c_{41},  c_{42}, 
c_{43},  c_{44},  c_{45},  c_{46},   c_{56},  c_{57},  c_{58} \}
$$

\medskip
{\it (2) The Generic Fiber.} Let $L = Q(K[C_0]) = K(c_{51},  c_{52},  c_{53},  c_{54},  c_{55})$,
and let $R_L= L[C_+] / J\,L[C_+]$ be the homogeneous coordinate ring 
of the generic fiber of~$\Theta$ (see Definition~\ref{def-GenFib}). 
A straightforward computation shows that $\Spec(R_L)$ is isomorphic 
to an affine space of dimension~10. 

\medskip
{\it (3) The $K[C_0]$-Linear Coefficient Matrix.} To go on, we first need a better system 
of generators of~$J$. The isomorphism $\Phi$ provides a tuple of $W'$-homogeneous generators 
of~$J$. Using a suitable truncated Gr\"obner basis, we can extract a subtuple~$F$ 
which consists of 15 polynomials and is a minimal $W'$-homogeneous set of generators of~$J$.

According to Remark~\ref{rem-blockmat}, the $K[C_0]$-linear coefficients matrix $\cmat_{K[C_0]}(F)$
can be arranged as a block matrix 
$$
\cmat_{K[C_0]}(F) \;=\; \begin{pmatrix} 
\cmat_{K[C_0]}^{(1)}(F) & 0 \\ 0 & \cmat_{K[C_0]}^{(2)}(F)  
\end{pmatrix}
$$
with the following two blocks:
$$
\cmat_{K[C_0]}^{(1)}(F)\tr =
\left( \begin{smallmatrix}
0  & -c_{51}c_{52} +c_{54}  & -c_{51}^2c_{52} +c_{52}c_{53} \\
c_{52}c_{53} -c_{51}c_{54} & c_{51}^2 -c_{53}  & c_{51}^3 -c_{51}c_{53} \\
0  & 0  & -c_{51}c_{52} +c_{54} \\
-c_{51}c_{52} +c_{54}  & 0  & c_{51}^2 -c_{53} \\
c_{51}^2 -c_{53}  & 0  & 0 \\
0  & -c_{52}^2 +c_{55}  & -c_{51}c_{52}^2 +c_{52}c_{54} \\
c_{52}c_{54} -c_{51}c_{55}  & c_{51}c_{52} -c_{54}  & c_{51}^2c_{52} -c_{51}c_{54} \\
0  & 0  & -c_{52}^2 +c_{55} \\
-c_{52}^2 +c_{55}  & 0  & c_{51}c_{52} -c_{54} \\
c_{51}c_{52} -c_{54}  & 0  & 0 \\
-c_{52}c_{54} +c_{51}c_{55}  & 0  & -c_{52}c_{53} +c_{51}c_{54} \\
c_{52}^2 -c_{55}  & 0  & c_{51}c_{52} -c_{54} \\
-c_{51}c_{52} +c_{54}  & 0  & -c_{51}^2 +c_{53} 
\end{smallmatrix} \right)
$$

$$ 
\cmat_{K[C_0]}^{(2)}(F) =
\left( \begin{smallmatrix}
0  & 0  & c_{52}^2 {-}c_{55}\\ 
0  & c_{52}^2 {-}c_{55}  & 0\\ 
0  & 0  & {-}c_{51}c_{52} {+}c_{54}\\
0  & {-}c_{51}c_{52} {+}c_{54}  & 0\\
0  & 0  & c_{51}^2 {-}c_{53}\\
0  & {-}c_{51}^2 {+}c_{53}  & 0\\
{-}c_{52}^2 {+}c_{55}  & 0  & c_{52}^3 {-}c_{52}c_{55}\\
c_{51}c_{52} {-}c_{54}  & c_{52}c_{53} {-}c_{51}c_{54}  & {-}c_{51}c_{52}^2 {+}c_{52}c_{54}\\
c_{51}^2 {-}c_{53}  & 0  & {-}c_{51}^2c_{52} {+}c_{52}c_{53}
\end{smallmatrix} \right)
$$
Note that there are 13 indeterminates of $W'$-degree one in~$Y$ and three $W'$-ho\-mo\-ge\-neous
polynomials of $W'$-degree one in~$F$ whose $K[C_0]$-linear parts
are in the rows of $\cmat_{K[C_0]}^{(1)}(F)$. Similarly, there are three indeterminates of $W'$-degree two in~$Y$
and nine $W'$-homogeneous polynomials of $W'$-degree two in~$F$ whose $K[C_0]$-linear parts are in
the rows of $\cmat_{K[C_0]}^{(2)}(F)$. 
As the origin of the generic fiber is smooth (see~(2)), we get $\rk(\cmat_{K[C_0]}(F)) = 16 - 10 = 6$
by Remark~\ref{rem-blockmat}.c.

\medskip
{\it (4) The Set of Zero Section Singularities.}
By Theorem~\ref{thm-CharSing0}.e,  we have
$\Sing_0(\Theta) = \mathcal{Z}(J_6)$, where $J_6$ is the ideal generated by the 
minors of order~6 of $\cmat_{K[C_0]}(F)$.

In view of the block structure of~$\cmat_{K[C_0]}(F)$, we let 
$J^{(1)}_3$ be  the ideal generated by the minors of order~3 of $\cmat_{K[C_0]}^{(1)}(F)$, 
we let $J^{(2)}_3$ be the ideal generated by the minors or order~3 of $\cmat^{(2)}_{K[C_0]}(F)$, and
we get $J_6 = J^{(1)}_3 \cdot J^{(2)}_3$. This implies
$$
\mathcal{Z}(J_6) \;=\; \mathcal{Z}(J^{(1)}_3) \cup \mathcal{Z}(J^{(2)}_3)
$$
When we compute $J^{(1)}_3$, we see that five of its generators are cubes.
Replacing them by their cube roots yields the ideal
$$
\mathfrak{p} \;=\; \langle c_{53} -c_{51}^2, \  c_{54} -c_{51}c_{52}, \ c_{55}  -c_{52}^2\rangle
\;=\; \Rad(J^{(1)}_3)
$$
where the second equality follows from the observation that $\mathfrak{p}$ is clearly 
a prime ideal. For the ideal $J^{(2)}_3$, a similar computation yields the surprising
result that $\Rad(J^{(2)}_3) = \mathfrak{p}$ again.

Altogether, we obtain $\Sing_0(\Theta) = \mathcal{Z}(\mathfrak{p})$, and
using the $(c_{53},c_{54},c_{55})$-separating reembedding, we see that $\Sing_0(\Theta)$
is isomorphic to an affine plane.

\medskip
{\it (5) The Set of Vertex Singularities.}
For every zero $\Gamma = (\gamma_1, \gamma_2, \gamma_1^2, \gamma_1\gamma_2, \gamma_2^2)
\in K^5$ of~$\mathfrak{p}$, the matrix $\cmat_K(F_\Gamma)$ turns out to
be the zero matrix.
According to Proposition~\ref{thm-CharSing_v}.c, every point
$\Gamma \in \Sing_0(\Theta) \setminus \Sing_v(\Theta)$ has to satisfy 
$0 = \rk(\cmat_K(F_\Gamma)) = n-m-d_{F_\Gamma} = 16- d_{F_\Gamma}$. 
As $\dim(\BO)=15$, this is impossible. Thus we conclude that 
$\Sing_v(\Theta) =\Sing_0(\Theta)$ here.

\medskip
{\it (6) The Generic Fiber over $\Sing_0(\Theta)$.}
Next we restrict the morphism $\Theta$ to $\Sing_0(\Theta)$ and
apply the $(c_{53},c_{54},c_{55})$-separating reembedding to the
restricted family. We obtain a positive $A' = K[c_{51},c_{52}]$-algebra
$R' = A'[C_+]/I'$ where the ideal~$I'$ is generated by 12 $W$-homogeneous polynomials.

If we let $Q=K(c_{51},c_{52})$, the generic fiber of~$\Theta$
over $\Sing_0(\Theta)$ has the homogeneous coordinate ring
$R'_Q= Q[C_+]/I'\,Q[C_+]$. Its dimension turns out to be $\dim(R'_Q)=11$
which is one more than the dimension of the generic fiber of~$\Theta$,
namely $\dim(R_L)=10$.

\medskip
{\it (7) The Set of Singular Fibers.}
The set $\Sing_s(\Theta)$ corresponds to the set of all points 
$\Gamma' = (\gamma_1,\gamma_2) \in \AA^2_K$ such that the singular locus of the fiber $R'_{F_{\Gamma'}}$
contains a point which is different from its vertex.

A direct computation shows that, for all $(\delta_1,\dots,\delta_6)\in K^6\setminus \{(0,\dots,0)\}$,
the point whose coordinates are 
\begin{align*}
(&c_{26}, c_{31}, &&c_{32},  c_{33}, &&c_{34}, c_{35}, &&c_{36}, c_{41}, &&c_{42}, c_{43}, &&c_{44}, 
c_{45}, &&c_{46}, c_{56}, &&c_{57}, c_{58}) \\
\;=\; (&\;\delta_1,\;\delta_4, &&\; \;0, \;2\delta_5, &&\;\delta_6, 
         \;0, &&\;\delta_2, \;0, &&\;\delta_4, \;0, &&\;\delta_5, \;2\delta_6, &&\;\delta_3, \;\delta_4, &&\;\delta_5,\;\delta_6)
\end{align*}
is a singular point of the fiber $F_{\Gamma'}$ which is different from its vertex.
Therefore we conclude that $\Sing_s(\Theta) = \Sing_0(\Theta) = \mathcal{Z}(\mathfrak{p})$.

\end{example}

\medskip
In the second subcase for $n=3$ and $\mu=5$, namely for $\OO = \{1,x,y,z,yz\}$, 
a somewhat different picture emerges.

\begin{example}{\bf (The $\{1,x,y,z,yz\}$-Border Basis Scheme)}\label{ex-13yz}$\mathstrut$\\
Let $K$ be a perfect field, let $P = K[x,y,z]$, and consider the order ideal
$\OO = \{t_1,\dots,t_5\}$, where $t_1=1$, $t_2=z$, $t_3=y$, $t_4=x$, and $t_5=yz$. 
Its border is given by $\partial\OO = \{b_1,\dots,b_8\}$, where $b_1=z^2$, $b_2=xz$, $b_3=y^2$,
$b_4=xy$, $b_5=x^2$, $b_6=yz^2$, $b_7=y^2z$, and $b_8=xyz$.
Hence we have $\mu=\#(\OO)=5$ and  $\nu=\#(\partial\OO)=8$ here.

Clearly, $\BO$ is a MaxDeg border basis scheme embedded in $\AA^{40}_K$.
Its coordinate ring $B_\OO = K[c_{11}, c_{12}, \dots, c_{58}] / I(\BO)$
is non-negatively graded by
\begin{align*}
W &\;=\; (2,  2,  2,  2,  2,  3,  3,  3,     1,  1,  1,  1,  1,  2,  2,  2,    1,  1,  1,  1,\\  
&\qquad\; 1,  2,  2,  2,     1,  1,  1,  1,  1,  2,  2,  2,     0,  0,  0,  0,  0,  1,  1,  1)
\end{align*}
We have $ K[C_0] = K[c_{51}, c_{52}, c_{53}, c_{54}, c_{55}]$ and  $B_\OO$
is a positive $K[C_0]$-algebra.
The ideal $I(\BO)$ is generated by~60 $W$-homogeneous polynomials.
According to Remark~\ref{rem-MaxDegBBS}, the scheme  $\BO$ is irreducible. 
Thus it coincides with its principal component, and therefore 
its dimension is $\dim(B_\OO) = 3 \cdot 5 = 15$.

\medskip
{\it (1) A $Z$-Separating Re-Embedding.}
In order to study the canonical morphism $\Theta\colon\BO \longrightarrow \AA^5_K$, its fibers
and its sets of singularities, we can perform a $Z$-separating re-embedding of~$\BO$,
because $I(\BO)\cap K[C_0] = \{0\}$ shows that the resulting isomorphism fixes the base space
$\AA^5_K = \Spec(K[C_0])$ and induces isomorphisms of the fibers and singular loci.

Using~\cite[Algorithm 4.2]{KR5}, we calculate the tuple 
$$
Z \;=\; (c_{11},  c_{12},  c_{13},  c_{14},  c_{15}, c_{16},  c_{17},  c_{18},   c_{21},  c_{22},  
c_{24},  c_{25}, c_{26}, c_{28},  c_{33}, c_{36},  c_{38},  c_{46},  c_{48})
$$
of~19 indeterminates which has the property that the ideal $I(\BO)$ is $Z$-separating. 
Hence there exists an epimorphism $\phi: K[C] \longrightarrow K[Y] $, where $Y=C\setminus Z$ 
is a set of 21 indeterminates, such that~$\phi$ induces an isomorphism of $K[C_0]$-algebras 
$\Phi \colon K[C]/I(\BO) \cong K[Y] / (I(\BO)\cap K[Y])$. 

Now we let $\M=\langle c_{ij}\rangle$ and compute $\dim_K(\Lin_\M(I(\BO))) = 25$.
Hence we have $\dim_K(K[C]/\Lin_\M(I(\BO)))=15$.
Together with $\dim(\BO) = 15$ and the fact that~$\BO$ is irreducible,
we obtain that the origin is a {\it regular}
point of~$\BO$ in stark contrast to the previous example.

Next, we apply the isomorphism~$\Phi$
and continue our investigation by examining the positive $K[C_0]$-algebra $R = K[Y]/J$, 
where  $J = I(\BO)\cap K[Y]$. Notice that the set of indeterminates~$Y$ satisfies
$Y=C_0 \cup C_+$ with $C_0 = \{ c_{51},  c_{52},  c_{53},  c_{54},  c_{55}\}$
and 
$$
C_+ \;=\; \{ c_{23},  c_{27},  c_{31},  c_{32},  c_{34},  c_{35},  c_{37},  c_{41},  c_{42}, 
c_{43},  c_{44},  c_{45},  c_{46},   c_{56},  c_{57},  c_{58} \}
$$
The degrees of the indeterminates in~$C_+$ are given by
$$
W' \;=\; ( 1, 2, 1, 1, 1, 1, 2, 1, 1, 1, 1, 1, 2, 1, 1, 1 )
$$

\medskip
{\it (2) The $K[C_0]$-Linear Coefficient Matrix.} 
As in the preceding example, we now compute a better system 
of generators of~$J$. The isomorphism $\Phi$ provides a tuple of $W'$-homogeneous generators 
of~$J$. Using a suitable truncated Gr\"obner basis, we can extract a subtuple~$F$ 
which consists of 15 polynomials and is a minimal $W'$-homogeneous set of generators of~$J$.

According to Remark~\ref{rem-blockmat}, the $K[C_0]$-linear coefficients matrix $\cmat_{K[C_0]}(F)$
can be arranged as a block matrix with the two blocks.
Since there are 13 indeterminates of $W'$-degree one in~$Y$ and three $W'$-ho\-mo\-ge\-neous
polynomials of $W'$-degree one in~$F$ whose ${K[C_0]}$-linear parts
are in the rows of $\cmat_{K[C_0]}^{(1)}(F)$, this matrix has size $13\times 3$.
Similarly, since there are three indeterminates of $W'$-degree two in~$Y$
and nine $W'$-homogeneous polynomials of $W'$-degree two in~$F$ whose $K[C_0]$-linear parts are in
the rows of $\cmat_{K[C_0]}^{(2)}(F)$, this matrix has size $3\times 9$.
Altogether, the matrix $\cmat_{K[C_0]}(F)$ has the form
$$
\cmat_{K[C_0]}(F) \;=\; \begin{pmatrix} 
\cmat_{K[C_0]}^{(1)}(F) & 0 \\ 0 & \cmat_{K[C_0]}^{(2)}(F)  
\end{pmatrix}
$$
and we get $\rk(\cmat_{K[C_0]}(F)) = 16-10= 6$  by Remark~\ref{rem-blockmat}.c.

\medskip
{\it (3) The Set of Zero Section Singularities.}
The computation proceeds exactly as in the preceding example.
Once again, the ideals generated by the minors of order~3 of the two blocks
have the same radical
$$
\mathfrak{p} \;=\; \langle c_{52} -c_{51}c_{54},\;  c_{55} -c_{51}c_{54}^2,\; 
c_{51}c_{53} -1 \rangle
$$
so that $\Sing_0(\Theta) = \mathcal{Z}(\mathfrak{p})$ is an irreducible
variety in $\AA^5_K$. More precisely, $\Sing_0(\Theta)$ is isomorphic 
to an open set of an affine plane
in $\Spec(K[C_0])= \AA_K^5$. As in Example~\ref{ex-13z2}, the generic fiber 
over $\Sing_0(\Theta)$ has dimension 11,
while the generic fiber of $\Theta$ has dimension 10.

\medskip
{\it (4) The Set of Vertex Singularities.}
For every zero $\Gamma\in K^5$ of~$\mathfrak{p}$, the matrix $\cmat_K(F_\Gamma)$ 
turns out to be the zero matrix. With the same proof as in the preceding example,
it follows that $\Sing_v(\Theta) =\Sing_0(\Theta)$.

\medskip
{\it (5) The Set of Singular Fibers}.
Let $\Gamma = (\gamma_1, \gamma_1 \gamma_2, \tfrac{1}{\gamma_1}, \gamma_2, \gamma_1 \gamma_2^2)$
be a point in $\Sing_0(\Theta)$, where $\gamma_1,\gamma_2\in K$ and $\gamma_1\ne 0$.
A direct computation shows that, for all $(\delta_1,\delta_2,\delta_3)\in K^3\setminus \{(0,0,0)\}$,
the point whose coordinates are 
\begin{align*}
(&c_{23}, c_{27}, &&c_{31}, c_{32}, &&c_{34}, c_{35}, &&c_{37}, c_{41}, &&c_{42}, 
c_{43}, &&c_{44}, c_{45}, &&c_{46}, c_{56}, &&c_{57}, c_{58}) \\
\;=\; (&\;0, \;\;\delta_1, &&\;\;0,\;\;0, &&\;\;0, \;\;0, &&\;\delta_2,\;\;0,  &&\;\;0,\;\;0,  
&&\;\;0, \;\;0, &&\;\delta_3,\;\;0, &&\;\;0, \;\;0)
\end{align*}
is a singular point of the fiber $F_\Gamma$ which is different from its vertex.
Therefore we conclude that $\Sing_s(\Theta) = \Sing_0(\Theta) = \mathcal{Z}(\mathfrak{p})$.

\end{example}

\bigbreak
%%%%%%%%%%%%%%%%%%%%%%%%%%%%%%%%%%%%%%%%%%%%%%
%
% Section 8: Conclusion and Outlook
%
%%%%%%%%%%%%%%%%%%%%%%%%%%%%%%%%%%%%%%%%%%%%%%

\section{Conclusion and Outlook}
\label{sec-conclusion}

In this paper we further studied  non-negatively graded algebras~$R =P/I$ whose
degree zero component~$P_0$ is a polynomial ring over a field. Using the corresponding
morphism $\Theta \colon \Spec(R) \To \Spec(P_0) \cong \AA^m_K$, we showed 
that $\Spec(R)$ is connected and that the fibers over a suitable open 
subset of~$\AA^m_K$ are affine spaces. 
Then the main emphasis was to consider various loci in~$\AA^m_K$ 
corresponding to certain types of singularities.

More precisely, we looked at $\Sing_0(\Theta)$, the set of points corresponding to singularities
of $\Spec(R)$ on the zero section, at $\Sing_v(\Theta)$, the set of points for which the fiber
has a singularity at its origin, and at $\Sing_s(\Theta)$, 
the set of points for which the corresponding fiber has  singular points other than its vertex, 
or equivalently, for which the associated weighted projective scheme has singularities.

We found characterizations of those loci in terms
of the ranks of a $P_0$-linear coefficient matrix and turned these characterizations into
algorithms for computing the loci.

All of these results were motivated by the wish to examine MaxDeg border basis schemes
and the difficulty posed by their high embedding dimensions, even after performing
the best possible $Z$-separating reembeddings. Motivated by this application scenario,
many natural questions arise, some of which could also be studied in the general context 
of positive algebras:

\begin{enumerate}
\item[(1)] Are MaxDeg border basis schemes always reduced? When are they irreducible?
(For large order ideals, some MaxDeg border basis schemes are known to be reducible, 
see~\cite{KR3},\cite{Iar}.)

\item[(2)] What can one say about the irreducible components of the fibers of~$\Theta$?

\item[(3)] What is the dimension and structure of the singular locus of the fibers~$F_\Gamma$
with $\Gamma \in \Sing_s(\Theta)$?

\item[(4)] In the setting of MaxDeg border basis schemes, consider the points in the zero section 
of~$\Theta$ corresponding to $\Sing_0(\Theta)$ or $\Sing_v(\Theta)$. Can the 0-dimensional schemes 
they represent be characterized geometrically?
The same question arises for the singular points of the fibers over $\Sing_s(\Theta)$.

\end{enumerate}

\section*{Acknowledgements.} 
The authors thank the referee for the very careful reading of the paper.

\bigbreak
%%%%%%%%%%%%%%%%%%%%%%%%%%%%%%%%%%%%%%
%
%   Bibliography
%
%%%%%%%%%%%%%%%%%%%%%%%%%%%%%%%%%%%%%%


\begin{thebibliography}{99}

\bibitem{ABR} J.\ Abbott, A.M.\ Bigatti, and L.\ Robbiano,
CoCoA: a system for doing Computations in Commutative Algebra, 2019,
Available at \url{https://sites.google.com/view/cocoa-cocoalib}


\bibitem{BR} M.\ Beltrametti  and L.\  Robbiano, Introduction to the
theory of weighted projective spaces, Expo. Math. {\bf 4} (1986), 111--162


\bibitem{DJNT} T.\ Douvropoulos, J.\ Jelisiejew, B.I.U.\ N{\o}dland, Z.\ Teitler,
The Hilbert scheme of 11 points is irreducible, in: G.\ Smith, B.\ Sturmfels (eds), 
{\it Combinatorial Algebraic Geometry}, Fields Institute Communications {\bf 80},
Springer, New York 2017, \\
\url{https://doi.org/10.1007/978-1-4939-7486-3_15}


\bibitem{Har} R.\ Hartshorne, Connectedness of the Hilbert scheme, 
Publ.\ Math.\ IHES {\bf 29} (1966), 7–-48, \url{https://doi.org/10.1007/BF02684803}


\bibitem{Iar} A.\ Iarrobino, Reducibility of the families of 0-dimensional
schemes on a variety,\  Invent.\ Math.\ {\bf 15} (1972), 72--77,
\url{https://doi.org/10.1007/BF01418644}


\bibitem{KLR0} M.\ Kreuzer, L.N.\ Long, and L.\ Robbiano, 
Computing subschemes of the border basis scheme, 
Int. J. Algebra Comput. {\bf 30} (2020), 1671--1716,\\
\url{https://doi.org/10.1142/S0218196720500599}


\bibitem{KLR1} M.\ Kreuzer, L.N.\ Long, and L.\ Robbiano, 
Cotangent spaces and separating re-embeddings, J.\ Algebra Appl. {\bf 21} (2022), 
\url{https://doi.org/10.1142/S0219498822501882}


\bibitem{KLR2} M.\ Kreuzer, L.N.\ Long, and L.\ Robbiano, 
Restricted Gr\"obner fans and re-embeddings of affine algebras, 
S\~ao Paulo J.\ Math.\ Sci.\ {\bf 17} (2023), 242--267,\\
\url{https://doi.org/10.1007/s40863-022-00324-w}


\bibitem{KLR3} M.\ Kreuzer, L.N.\ Long, and L.\ Robbiano, 
Re-embeddings of affine algebras via Gr{\"o}bner fans of linear ideals, 
Beitr.\ Algebra Geom.\ {\bf 65} (2024), 827--851,\\
\url{https://doi.org/10.1007/s13366-024-00733-2}


\bibitem{KR1} M.\ Kreuzer and L.\ Robbiano, {\it Computational
Commutative  Algebra 1}, Springer-Verlag, Berlin Heidelberg, 2000,
\url{https://doi.org/10.1007/978-3-540-70628-1}


\bibitem{KR2} M.\ Kreuzer and L.\ Robbiano, {\it Computational
Commutative Algebra 2}, Springer-Verlag, Berlin Heidelberg, 2005,
\url{https://doi.org/10.1007/3-540-28296-3}


\bibitem{KR3} M.\ Kreuzer and L.\ Robbiano, Deformations of border bases,
Collect.\ Math.\ {\bf 59} (2008), 275--297,
\url{https://doi.org/10.1007/BF03191188}


\bibitem{KR4} M.\ Kreuzer and L.\ Robbiano, The geometry of border bases,
J.\ Pure Appl.\ Algebra {\bf 215} (2011), 2005--2018,
\url{https://doi.org/10.1016/j.jpaa.2010.11.011}


\bibitem{KR5} M.\ Kreuzer and L.\ Robbiano, Elimination by substitution,
J. Symbolic Comput. {\bf 131} (2025) 102445,
\url{https://doi.org/10.1016/j.jsc.2025.102445}


\bibitem{KR6} M.\ Kreuzer and L.\ Robbiano, Re-Embeddings of Special Border Basis Schemes,
In J.\ Brennan, A.\ Simis (eds), {\it Commutative Algebra,
The Mathematical Legacy of Wolmer V. Vasconcelos},
De Gruyter, pp 353--388 (2025),
\url{ https://doi.org/10.1515/9783110999365-012}


\bibitem{Kun} E.\ Kunz, {\it Introduction to Commutative Algebra and Algebraic Geometry},
Birkh\"auser, Boston 1985, \url{https://doi.org/10.1007/978-1-4614-5987-3}


\bibitem{Maz} G.\ Mazzola, Generic finite schemes and Hochschild cocycles, 
Comment. Math. Helv. {\bf 55} (1980), 267--293,
\url{https://doi.org/10.1007/BF02566686}


\bibitem{PS} I.\ Peeva and M.\ Stillman, Connectedness of Hilbert schemes, J.\ Algebraic Geom. {\bf 14} 
(2005), 193--211, \url{https://doi.org/10.1090/S1056-3911-04-00386-8}


\bibitem{Rob1} L.\ Robbiano, On border basis and Gr\"obner basis schemes,
Collect.\ Math.\ {\bf 60} (2009), 11--25, \url{https://doi.org/10.1007/BF03191213}


\bibitem{Rob2} L.\ Robbiano, Border bases and border basis schemes, 
Galois J.\ Algebra {\bf 2} (2026), 66–-152,
\url{https://doi.org/10.65908/gja.2026.8575.1063}


\bibitem{Wei} V.\ Weispfenning, Comprehensive Gr\"obner bases, J.\ Symb.\ Comput.\ {\bf 14} (1992), 1--29,\\
 \url{https://doi.org/10.1016/0747-7171(92)90023-W}

\end{thebibliography}
\end{document}